\documentclass[a4paper,11pt]{article}
\usepackage{amsthm,amsmath,amssymb}
\usepackage{mathrsfs}
\usepackage[dvipdfmx]{graphicx}
\usepackage{here}
\usepackage{color}
\usepackage{braket}
\usepackage[
 bookmarks=true,%
 bookmarksnumbered=true,%
 colorlinks=true,%
 linkcolor=blue,%
 setpagesize=false,%
 ]{hyperref}
\usepackage{newpxtext,newpxmath}

\numberwithin{equation}{section}

\theoremstyle{plain}
\theoremstyle{definition}
\newtheorem{definition}{Definition}[section]
\newtheorem{theorem}[definition]{Theorem}
\newtheorem{proposition}[definition]{Proposition}

\newtheorem{lemma}[definition]{Lemma}
\newtheorem{remark}[definition]{Remark}

\newenvironment{thm}{\begin{theorem}}{\end{theorem}}

\newcommand{\Real}{\mathbb{R}}

\newcommand{\EXP}[1]{\mathsf{E}\left[#1\right]}
\newcommand{\PROB}[1]{\mathsf{P}\left[#1\right]}

\newcommand{\hd}{H\" older }

\newcommand{\D}{\mathrm{d}}

\usepackage[utf8]{inputenc}

\title{A partial rough path space for rough volatility}
\author{Masaaki Fukasawa and Ryoji Takano
\\ {\small \textit{Graduate School of Engineering Science, Osaka University}}}
\date{}

\begin{document}
\maketitle

\begin{abstract}
  We develop a variant of rough path theory tailor-made for analyzing a class of financial asset price models known as rough volatility models. As an application, we prove a pathwise large deviation principle (LDP) for a certain class of rough volatility models, which in turn describes the limiting behavior of implied volatility for short maturity under those models. First, we introduce a partial rough path space and an integration map on it and then investigate several fundamental properties including local Lipschitz continuity of the integration map from the partial rough path space to a rough path space. Second, we construct a rough path lift of a rough volatility model. Finally, we prove an LDP on the partial rough path space, and the LDP for rough volatility then follows by the continuity of the solution map of rough differential equations.  
\end{abstract}

\newcommand{\ABS}[1]{\left(#1\right)} 
\newcommand{\veps}{\varepsilon} 


\section{Introduction}
A rough volatility model is a stochastic volatility model for an asset price process with volatility being rough, meaning that the H\"older regularity of the volatility path is less than half. Recently, such models have been attracting attention in mathematical finance because of their unique consistency to market data. Indeed, rough volatility models are the only class of continuous price models that are consistent to a power law of implied volatility term structure typically observed in equity option markets, as shown by \cite{F21}. 
One way to derive the power law under rough volatility models is to prove a large deviation principle (LDP) as done by many authors \cite{FoZh,BaFrGuHoSt, BaFrGaMaSt,  FrGaPi1, FrGaPi2,Gu, GuViZh1,GuViZh2, jacquier, GeJaPaStWa, JaPa, Gu22} using various methods.
An introduction to LDP and some of its applications to finance and insurance problems can be found in \cite{Pham, FGGJT}.
In the context of the implied volatility, a short-time LDP under local volatility models provides a validity proof for a precise approximation known as the BBF formula~\cite{BBF, ABBF}. The SABR formula, which is of daily use in financial practice, is also proved as a valid approximation under the SABR model by means of LDP~\cite{Osa}.
From these successes in classical (non-rough) volatility models, we expect LDP for rough volatility models to provide in particular a useful implied volatility approximation formula for financial practice such as model calibration.

For the classical models that are described by standard stochastic differential equations (SDEs), an elegant way to prove an LDP is to apply the contraction principle in the framework of rough path analysis~\cite{FV, Friz}.
Under rough volatility models, the volatility of an asset price has a lower H\"older regularity than the asset price process. 
The stochastic integrands are therefore not controlled by the stochastic integrators in the sense of  \cite{Gubinelli}.
Hence, a rough volatility model is beyond the scope of rough path theory, which motivated \cite{BaFrGaMaSt} to develop a regularity structure for rough volatility.
For classical SDEs, the Freidlin--Wentzell LDP can be obtained as a consequence of the continuity of the solution map (the Lyons--It\^o map) that is the core of rough path theory.
In \cite{BaFrGaMaSt}, the LDP for rough volatility models is obtained using the continuity of Hairer's reconstruction map.
Herein, we take an approach that is similar to that of \cite{BaFrGaMaSt} in spirit but differs somewhat. Instead of embedding a rough volatility model into the abstract framework of regularity structure, we develop a minimal extension of rough path theory to incorporate rough volatility models.
Besides the relatively elementary construction, an advantage of our theory is that 
it ensures the continuity of the integration map between rough path spaces, which enables us to treat a more general model than \cite{BaFrGaMaSt}.

We focus on a model of the following form:
\begin{equation}\label{rvm}
\D S_t =  \sigma(S_t) f(\hat{X}_t,t) \D X_t, \quad S_0 \in \mathbb{R},
\end{equation}
where $X$ is a $d$-dimensional Brownian motion, $\hat{X}$ is an $e$-dimensional stochastic process of which components include $\int_0^t \kappa (t-s) \D X_s$
with a deterministic $L^2$ kernel $\kappa$. The stochastic integration is in the It\^o sense.
An example is the rough Bergomi model ($\kappa = \kappa_H$ is the Riemann--Liouville kernel (\ref{RLk}), $f$ is exponential, and $\sigma(s) =s$ in (\ref{rvm})) introduced by \cite{BaFrGa}. 
When $\kappa=\kappa_H$ or more generally $\kappa$ has a similar singularity to $\kappa_H$ with $H<1/4$, beyond the case of $\sigma(s) = 1$ or $\sigma(s) = s$, no LDP is available in the literature so far, including~\cite{BaFrGaMaSt}.
As mentioned above, the difference between classical SDEs and (\ref{rvm}) is that the volatility process $\hat{X}$ is not controlled by $X$ because of its lower regularity. 
From empirical evidence, we are particularly interested in the case where $\hat{X}$ is correlated with $X$ and $H<1/4$~\cite{GaJaRo, Ben, FuTaWe, Chr}. 
Unfortunately, the application of existing rough path theory involves iterated integrals of $\hat{X}$ while, as is well-known, the standard rough path lift of $(X,\hat{X})$ that is amenable to LDP
does not work when $H<1/4$; see e.g.,~\cite{Friz}.

Our idea, inspired by \cite{BaFrGaMaSt}, is to consider a partial rough path space in which we lack the iterated integrals of $\hat{X}$ but are still able to treat \eqref{rvm}. 
More precisely, we define the space of a triplet of iterated integrals driven by $X$ (we do not consider iterated integrals driven by $\hat{X}$) and rederive analytical results obtained in existing rough path theory. The notion of a partial rough path was introduced in \cite{GuLe} to prove the existence of global solutions for differential equations driven by a rough path with vector fields of linear growth. Our motivation is different and requires a space of higher-level paths. 
In contrast to \cite{BaFrGaMaSt},
our method does not rely on the theory of regularity structure and enables us to treat not only the rough Bergomi model but also the following rough volatility models:
\begin{enumerate}
\item[-] the rough SABR model~\cite{FuHoTa,Mu, FuGa,Fu}; 
\item[-] the mixed rough Bergomi model  \cite{BoDeGo};
\item[-] rough local stochastic volatility \cite{JaOu};
\item[-] the two-factor fractional volatility model \cite{FuKi}.
\end{enumerate}
To the best of our knowledge, no LDP for these models is established so far in the literature.

To explain the idea of the partial rough path,
here, we argue for how such a partial rough path space should be. Suppose that $d,e \geqq 1$, $x : [0,T] \rightarrow \mathbb{R}^d$, $\hat{x} : [0,T] \rightarrow \mathbb{R}^e$, and $f :\mathbb{R}^e \rightarrow \mathbb{R}$ are good enough. By the Taylor expansion, for $s< t$ (which are close enough), we have 
\begin{equation*}
\int_{s}^{t} f(\hat{x}_r) 
\D x_r 
\approx  f(\hat{x}_s) (x_t-x_s) + \sum_{|i| \leqq n} \frac{1}{i!} \partial^i f (\hat{x}_s) \left [ \int_s^t (\hat{x}_r - \hat{x}_s)^{i} \D x_r \right ]
\end{equation*}
and
\begin{equation*}
\begin{split}
 &\int_s^t \left ( \int_s^{r} \D y_u \right)\otimes \D y_r 
 \\ &\approx  \sum_{|j+k| \leqq n}  \frac{1}{j! k!}  \partial^j f(\hat{x}_s) \partial^k f(\hat{x}_s) \left [ \int_s^t (\hat{x}_r - \hat{x}_s )^k \left ( \int_s^r (\hat{x}_u - \hat{x}_s )^j  \D x_u \right)   \otimes \D x_r \right ],
\end{split}
\end{equation*} 
where $y_t := \int_{0}^{t} f( \hat{x}_r) \D x_r$, $i,j,k$ are multi-indices, and we use the following notation: 
\[
|i| := \sum_{l=1}^e i_l , 
\quad i! := \prod_{l=1}^e i_l!, 
\quad x^i := \prod_{l=1}^e (x_l)^{i_l}, 
\quad \partial^{i} : = \prod_{l=1}^e \left (\frac{\partial}{\partial x_l} \right)^{i_l}
\]
for $i=(i_1,...,i_e), \ x = (x_1,...,x_e)$.
Therefore, following the idea of rough path theory, 
we would be able to define a rough path integral $\int f(\hat{x}_r)  \D x_r$
if we could define 
\[
X^{(i)}_{st} : = \frac{1}{i!} \int_{s}^{t} (\hat{X}_{sr})^i  \D x_r
, \quad \mathbf{X}^{(jk)}_{st} := \frac{1}{k!} \int_{s}^{t} (\hat{X}_{sr})^{k} X^{(j)}_{sr} \otimes  \D x_r
\]
for $\hat{X}_{sr} : = \hat{x}_r  - \hat{x}_s$.
By the linearity of the integration and the binomial theorem (see Section~8.1 in \cite{Foll}), $X^{(i)}$ and $\mathbf{X}^{(jk)}$ should satisfy the following formulas respectively:
for any $i,j,k \in \mathbb{Z}^e_{+}$ and $s \leqq u \leqq t$, 
\begin{equation}\label{chen1}
X^{(i)}_{st} = X^{(i)}_{su} + \sum_{p \leqq i} \frac{1}{(i-p)!} (\hat{X}_{su})^{i-p} X^{(p)}_{ut}
\end{equation}
and
\begin{equation}\label{chen2}
\begin{split}
\mathbf{X}^{(jk)}_{st} 
= \mathbf{X}^{(jk)}_{su}  & + \sum_{q \leqq k} \frac {1}{(k-q)!}  (\hat{X}_{su} )^{k-q} X^{(j)}_{su} \otimes X^{(q)}_{ut}
\\ &+\sum_{p \leqq j} \sum_{q\leqq k} \frac{1}{(j-p)! (k-q)!}  (\hat{X}_{su} )^{j+k-p-q} \mathbf{X}^{(pq)}_{ut},
\end{split}
\end{equation}
where, for $i,j \in \mathbb{Z}^e_{+}$, $i\leqq j$ means for all $l \in \{1,...,e\}$, $i_l \leqq j_l$, and  $\mathbb{Z}_+$ is the set of the nonnegative integers. 
Our partial rough space is a space for $\hat{X}$, $X^{(i)}$ and $\mathbf{X}^{(jk)}$, where
the formulas \eqref{chen1} and \eqref{chen2}
should play the role of Chen's identity. 

In Section~2, we formulate such a partial rough path space and state some fundamental properties including the continuity of the integration map. In Section~3, we construct a rough path lift of our rough volatility model and state an LDP. Proofs are relegated to Section~4.

\section{A partial rough path space}
\subsection{Definition}
Throughout this article, we fix $\alpha \in  (\frac{1}{3},\frac{1}{2}]$, $\beta \in  (0,\frac{1}{2})$, $T>0$ and denote 
\[
\Delta_T := \{(s,t) | 0 \leqq s \leqq t \leqq  T \}, \quad I:= \{ i \in \mathbb{Z}^e_+ | \  |i| \beta +\alpha \leqq 1 \},
\]
and 
\[
J := \{ (j,k) \in \mathbb{Z}^e_+ \times \mathbb{Z}^e_+ | \  |j+k|\beta +2\alpha \leqq 1 \}.
\]
Extending the notion of an $\alpha$-\hd rough path in rough path theory,
here we define an $(\alpha,\beta)$ rough path.
\begin{definition}
An $(\alpha,\beta)$ rough path $\mathbb{X} = \left (\hat{X}, X^{(i)}, \mathbf{X}^{(jk)} \right)_{i \in I , (j,k)\in J}$ is a triplet of functions on $\Delta_T$ satisfying the following conditions for any $i \in I, (j,k) \in J$, and $s \leqq u \leqq t$.
\begin{enumerate}
\item[(i)] $\hat{X}$ is $\mathbb{R}^{e}$-valued, $X^{(i)}$ is $\mathbb{R}^d$-valued, and $\mathbf{X}^{(jk)}$ is $\mathbb{R}^d \otimes \mathbb{R}^d$-valued.
\item[(ii)] \textit{Modified Chen's relation}:  $\hat{X}_{st} = \hat{X}_{su} + \hat{X}_{ut}$, 
and $ X^{(i)}$ and $ \mathbf{X}^{(jk)}$ satisfy ($\ref{chen1}$) and ($\ref{chen2}$), respectively.
\item[(iii)] \textit{\hd regularity}:
\begin{equation*}
|\hat{X}_{st}| \lesssim |t-s|^{\beta}, \quad |X^{(i)}_{st}| \lesssim |t-s|^{|i|\beta + \alpha}, \quad |\mathbf{X}^{(jk)}_{st}| \lesssim |t-s|^{|j+k|\beta + 2\alpha}.
\end{equation*}
\end{enumerate}
Let $\Omega_{(\alpha ,\beta) \text{-Hld}}$ denote the set of $(\alpha,\beta)$ rough paths.
We define a metric function $d_{(\alpha,\beta)}$ on $\Omega_{(\alpha,\beta) \text{-Hld}}$ and a homogeneous norm $|||\mathbb{X}|||_{(\alpha,\beta)}$ respectively by
\begin{align*}
&d_{(\alpha,\beta)} (\mathbb{X},\mathbb{Y}) \\
&:=  ||\hat{X} - \hat{Y}||_{\beta \text{-Hld}} + \! \sum_{i \in I , (j,k) \in J} ||X^{(i)} - Y^{(i)}||_{|i|\beta + \alpha \text{-Hld}} + ||\mathbf{X}^{(jk)} - \mathbf{Y}^{(jk)}||_{|j+k|\beta +2\alpha \text{-Hld}}
\end{align*}
and
\begin{align*}
&|||\mathbb{X}|||_{(\alpha, \beta)} \\
&:=  ||\hat{X}||_{\beta \text{-Hld}} + \sum_{i \in I , (j,k) \in J} \left ( ||X^{(i)}||_{|i|\beta+\alpha \text{-Hld}} \right )^{1/(|i|+1)} + \left ( ||\mathbf{X}^{(jk)}||_{|j+k|\beta +2\alpha \text{-Hld}} \right)^{1/(|j+k|+2)},
\end{align*} 
where $\| \cdot \|_{\gamma \text{-Hld}}$ is the $\gamma$-\hd norm for two-parameter functions for $\gamma \in (0,1]$:
 \[
||X||_{\gamma\text{-Hld} } :=   \sup_{0\leqq s<t \leqq T} \frac{|X_{st}|}{|t-s|^{\gamma}}.
\]
\end{definition}

\begin{remark}
The modified Chen's relation and the \hd regularity of
$X^{(i)}$ and $\mathbf{X}^{(jk)}$ are from the following correspondence:
\begin{equation*}
X^{(i)}_{st} \leftrightarrow \frac{1}{i!} \int_{s}^{t} \left ( \hat{X}_{sr} \right)^{i} \D X^{(0)}_r, \quad \mathbf{X}^{(jk)}_{st} \leftrightarrow \frac{1}{k!} \int_{s}^{t} \left (\hat{X}_{sr} \right)^{k} 
X^{(j)}_{sr}  \otimes \D X^{(0)}_r
\end{equation*}
when $X^{(0)}$ and $\hat{X}$ have \hd regularity $\alpha$ and $\beta$, respectively.
Note also that $\left(X^{(0)},\mathbf{X}^{(00)}\right)$ is an $\alpha$-\hd rough path with the first level
$X^{(0)}$ and the second level $ \mathbf{X}^{(00)}$ in the usual rough path terminology. An $(\alpha,\beta)$ rough path has two first-level paths: $X^{(0)}$ and $\hat{X}$. 
\end{remark}

\begin{remark}
Our modified Chen's relation is a particular form of the algebraic structure of branched rough paths studied in \cite{Gubi2}. However, because $\hat{X}$ is not a controlled path of $X$, the novel framework of $(\alpha,\beta)$ rough paths is essential for establishing the rough path integral stated in the Introduction. 
\end{remark}

\begin{remark}[A comparison with \cite{BaFrGaMaSt}]
The iterated integral $X^{(i)}_{st} = \frac{1}{i!}\int_{s}^{t} \left ( \hat{X}_{sr} \right)^{i} \D X^{(0)}_r$
plays a key role also in \cite{BaFrGaMaSt} (see section 3.1 in \cite{BaFrGaMaSt}, where $X^{(i)}_{st} = \mathbb{W}^{i}_{st}$ in their notation). In \cite{BaFrGaMaSt}, its derivative $\frac{\D}{\D t} \mathbb{W}^{i}_{st}$ appears in the structure space of  regularity structure.
Our $(\alpha,\beta)$ rough path consists of not only $X^{(i)}_{st}$ but also $\mathbf{X}^{(jk)}_{st}$. The latter is required to construct a rough path integral as an element of a rough path space, while in \cite{BaFrGaMaSt} the corresponding integral is constructed as merely a distribution and such terms as $\mathbf{X}^{(jk)}_{st}$ are not necessary for that purpose.
As mentioned in Introduction, the key to treat \eqref{rvm} with a general function $\sigma$ is to construct $\int f(\hat{X}_t,t) \D X_t$ as an element of a rough path space.
\end{remark}

\subsection{$(\alpha,\beta)$ rough path integration}

Extending the rough path integration, here we introduce an integration with respect to an $(\alpha,\beta)$ rough path.
 \begin{definition}
Fix $\mathbb{X} \in \Omega_{(\alpha,\beta) \text{-Hld}}$. We define $Y^{(1)}$ and $Y^{(2)}$ as follows if they exist:
\begin{equation*}
Y^{(1)}_{st} := \lim_{|\mathcal{P}| \searrow 0} \sum_{p=1}^{N} \sum_{i \in I} \partial^i f(\hat{x}_{t_{p-1}}) X^{(i)}_{t_{p-1}t_{p}}, 
\end{equation*}
\begin{equation*}
Y^{(2)}_{st}  := \lim_{|\mathcal{P}| \searrow 0} \sum_{p=1}^{N} \left ( Y^{(1)}_{t_{0} t_{p-1}} \otimes Y^{(1)}_{t_{p-1} t_{p}}  +\!\!\!\!\! \sum_{ (j,k) \in J} \!\! \partial^j f(\hat{x}_{t_{p-1}}) \partial^k f(\hat{x}_{t_{p-1}}) \mathbf{X}^{(jk)}_{t_{p-1}t_{p}} \right ),
\end{equation*}
where
$\hat{x}_s := \hat{X}_{0s}$, and $\mathcal{P} = \{ s= t_0 < t_1 < ... <t_N = t   \}$ is a partition of the interval $[s,t]$. 
The mesh size $|\mathcal{P}|$ is defined by $|\mathcal{P}| = \max_p |t_p-t_{p-1}|$.
If they exist on $\Delta_T$, we denote $(Y^{(1)}, Y^{(2)})$ by $\int f(\hat{\mathbb{X}}) \D {\mathbb{X}} $, and we call this the $(\alpha, \beta)$ rough path integral of $f$.\\
\end{definition}

Denote by
$\Omega_{\alpha \text{-Hld}}$ the $\alpha$-H\"older rough path space,
and denote by $d_{\alpha}$  the metric function on $\Omega_{\alpha \text{-Hld}}$; see \cite{FH}, for example. 
Here, we state our first main result, 
the proof of which is given in Section~\ref{pr1}.

\begin{theorem}\label{reconst}
Let $n:= \max \{|i| : i \in I\}$ and assume that $f : \mathbb{R} \rightarrow \mathbb{R}$ is $C^{n+2}$. 
\begin{enumerate}
\item[(i)] For any $\mathbb{X} \in \Omega_{(\alpha,\beta) \text{-Hld}}$, the $(\alpha, \beta)$ rough path integral $\int f(\hat{\mathbb{X}}) \D \mathbb{X}$  is well-defined, and  $\int f(\hat{\mathbb{X}}) \D \mathbb{X} \in \Omega_{\alpha \text{-Hld}}$.

\item[(ii)] The integration map $\int : \Omega_{(\alpha,\beta) \text{-Hld}} \rightarrow \Omega_{\alpha \text{-Hld}}$ is locally Lipschitz continuous. More precisely, for any $M >0$, the map $\int |_{\mathcal{E}_M}$, restricted on the set 
\begin{equation*}
\mathcal{E}_M := \left \{ \mathbb{X} \in \Omega_{(\alpha,\beta) \text{-Hld}} | \ \  |||\mathbb{X}|||_{(\alpha, \beta)} \leqq  M \right \},
\end{equation*}
is Lipschitz continuous; that is, there exists a positive constant $C>0$ such that 
\begin{equation*}
 d_{\alpha} \left (\int f(\hat{\mathbb{V}}) \D \mathbb{V}, \int f(\hat{\mathbb{W}}) \D \mathbb{W} \right) \leqq C d_{(\alpha, \beta)} \left (\mathbb{V},\mathbb{W} \right), \quad  \mathbb{V},\mathbb{W} \in \mathcal{E}_M.
\end{equation*}
\end{enumerate}
\end{theorem}

\section{Large deviation}
\subsection{A lift to the partial rough path space}
We now construct an $(\alpha, \beta)$ rough path, which plays an important role in this paper. 
For notational simplicity we focus on a low dimensional case
(both $\kappa$ and $W$ below are one-dimensional)
but extensions to higher dimensional cases are straightforward.
The proof is deferred to Section~\ref{proof_lift}.
Let $\kappa :(0,T] \rightarrow [0 , \infty) $ as
\begin{align*}
    \kappa (t) := g(t) t^{\zeta - \gamma}, \quad t \in (0,T], 
\end{align*}
where $\gamma,\zeta \in(0,1)$ and $g$ is a Lipschitz function. For example, the Riemann--Liouville kernel 
\begin{align}\label{RLk}
      \kappa_H (t) := \frac{t^{H-1/2}}{\Gamma{(H+1/2)}}, \quad  t \in (0,T], \ H \in (0,1/2)
\end{align}
has the above form ($\zeta = H - \delta, \gamma = 1/2 - \delta$, $g(t) = 1/ \Gamma{(H+1/2)} $, where $\delta \in (0,1/2)$). 
For $\alpha \in (0,1]$, let $C^{\alpha \text{-Hld}}$ denote the space of $\alpha$-\hd continuous functions on $[0,T]$.
Let $\mathcal{K}: C^{\gamma \text{-Hld}} \rightarrow C^{\zeta \text{-Hld}}$ as
\begin{align*}
   \mathcal{K} f(t) 
& := \lim_{\epsilon \to 0} \left \{ \left[\kappa(t-\cdot) (f(\cdot) - f(t)) \right]^{t-\epsilon}_0 
+ \int_0^{t-\epsilon} (f(s) - f(t)) \kappa'(t-s) \D s \right\}\\
& = \kappa(t) (f(t) - f(0)) + \int_0^t (f(s)- f(t) ) \kappa' (t-s) \D s. 
\end{align*}

\begin{proposition}\label{lift}
Let $(\Omega, \mathcal{F}, \mathbb{P}, \{ \mathcal{F}_t \}_{t\geqq 0})$ be a filtered probability space, and fix $\alpha \in (1/3,1/2]$, $\beta \in (0,1/2)$, and $\gamma, \zeta \in (0,1)$ with $\gamma < 1/2$, $\beta < \zeta $. Suppose that $X = (X^1, ... ,X^d)$ is a $d$-dimensional (possibly correlated) Brownian motion, and $W$ is a one-dimensional Brownian motion possibly correlated to $X$.
Using the It\^o integration, define $\hat{X}$, $X^{(i)}$, and $\mathbf{X}^{(jk)}$ as follows: for $(s,t) \in \Delta_T$, $i \in I$ and $(j,k) \in J$, 
\begin{equation*}
\hat{X}^{(1)}_{st} := \mathcal{K} W(t) - \mathcal{K} W(s),
\end{equation*}
\begin{align*}
\hat{X}^{(2)}_{st} := t^{\zeta} -s^{\zeta},  
\end{align*}
\begin{equation*}
X^{(i)}_{st} := \frac{1}{i!} \int_{s}^{t} \left (\hat{X}_{sr}  \right)^i  \D X_r, \quad 
\mathbf{X}^{(jk)}_{st} := \frac{1}{k!} \int_{s}^{t}  \left (\hat{X}_{sr}  \right)^k X^{(j)}_{sr} \otimes \D X_r.
\end{equation*}
Let $\kappa_{st}(r) := \left( \kappa(t-r)  - \kappa(s-r)  1_{(0,s)}(r) \right) 1_{(0,t)} (r)$ and assume that
\[
||\kappa_{st} ||^2_{L^2(\mathbb{R}_+)} \lesssim C |t-s|^{2 (\zeta - \gamma) +1 }.
\]
Then we have the following.
\begin{enumerate}
\item[(i)] For a.s.\ $\omega \in \Omega$, $\mathbb{X}(\omega) := \left ( \hat{X}(\omega), X^{(i)}(\omega), \mathbf{X}^{(jk)}(\omega) \right)_{i \in I, (j,k) \in J}$ is an $(\alpha, \beta)$ rough path.
\item[(ii)]
It holds that 
\begin{equation*}
\left( \int f(\hat{\mathbb{X}}) \D {\mathbb{X}} \right )_{0t}^{(1)} = \int_{0}^{t} f(\hat{X}_{0r})  \D X_r \quad a.s., 
\end{equation*} 
where the left-hand side is the first level of the $(\alpha,\beta)$ rough path integral and the right-hand side is the It\^o  integral.
\end{enumerate}
\end{proposition}

\subsection{The large deviation principle on $\Omega_{(\alpha,\beta) \text{-Hld}}$}

We now discuss the LDP on $\Omega_{(\alpha,\beta) \text{-Hld}}$.
Following \cite{jacquier, GeJaPaStWa}, we use Garcia's theorem ~\cite{garcia}.
Let $(W, W^{\perp})$ be a two-dimensional standard Brownian motion and $X := \rho W + \sqrt{1-\rho^2} W^{\perp}, \ \rho \in [-1,1]$.
Define $\hat{X}, X^{(i)},\mathbf{X}^{(jk)}$ as in Proposition~\ref{lift} with $d =1$, $e=2$ . 
We state our second main result, the proof of which is given in Section~\ref{pr3}.

\begin{thm}\label{ldp}
 Let $\mathbb{X} =(  \hat{X}, X^{(i)}, \mathbf{X}^{(jk)})$ be the random variable taking values on ($\Omega_{(\alpha,\beta) \text{-Hld}},d_{(\alpha,\beta)}$) defined as above.  
 Then, the sequence of triplets 
 $$\mathbb{X}^{\epsilon} :=\left( \epsilon^{1/2} \hat{X}, \epsilon^{(|i|+1)/2} X^{(i)}, \epsilon^{(|j+k|+2)/2}\mathbf{X}^{(jk)}\right) $$ 
 satisfies the LDP on  ($\Omega_{(\alpha,\beta) \text{-Hld}},d_{(\alpha,\beta)}$) with speed $\epsilon^{-1}$ with good rate function 
\begin{align*} 
\tilde{I}^{\#\#}(\hat{x}, x^{(i)}, \mathbf{x}^{(jk)}) 
: = \inf \left \{\tilde{I}^{\#} (\tilde{v}) \middle| \tilde{v} \in \mathcal{H}, \  (\hat{x},x^{(i)}, \mathbf{x}^{(jk)} ) = \mathbb{L} \circ \mathbb{K}(\tilde{v}) \right \},
\end{align*}
where $\mathcal{H}$ is the Cameron--Martin space from $[0,T]$ to $\mathbb{R}^2$,  
\[
\mathbb{K} (\tilde{v}) 
: = \left( \left( \int_0^{\cdot} \kappa(\cdot - r ) \D \tilde{v}^{(1)}_r, 0\right), \rho \tilde{v}^{(1)} + \sqrt{1-\rho^2} \tilde{v}^{(2)} \right)
\]
and 
\[
\mathbb{L}(u,v) := ( \delta u, u \cdot v, u \ast v ), \quad u,v \in C_{[0,T]}, \ v \in  \mathcal{H},
\]
$\delta u_{st} := u_t - u_s$, $u \cdot v = (u \cdot_i v)$, $u \ast v = (u \ast_{jk} v)$, and 
\[
(u \cdot_i v)_{st} := \int_s^t (u_r-u_s)^i \D v_r, \quad (u \ast_{jk} v)_{st} := \int_s^t (u\cdot_j v)_{sr} (u_r-u_s)^k \D v_r.
\]
Here, $\tilde{I}^{\#}:C\rightarrow [0,\infty)$ is the rate function of two-dimensional Brownian motion:
\[
\tilde{I}^{\#} (\tilde{v}) : = \begin{cases}
\frac{1}{2} || \tilde{v} ||^2_{\mathcal{H}}, &  \tilde{v} \in \mathcal{H}, \\
\infty, &  \text{otherwise}.
\end{cases}
\]
\end{thm}

\begin{theorem}\label{thm:34}
The sequence of the processes $\left\{Y^{\epsilon} := \int f(\hat{\mathbb{X}}^{\epsilon})\D \mathbb{X}^{\epsilon} \right\}_{\epsilon \geqq 0}$ satisfies the LDP on $(\Omega_{\alpha \text{-Hld}}, d_{\alpha})$ with speed $\epsilon^{-1}$ with good rate function
\begin{align*}
\tilde{I}^{\#\#\#} (y) 
&:= \inf \left \{ \tilde{I}^{\#\#} (\mathbb{X}) \middle|\  \mathbb{X} \in \Omega_{(\alpha,\beta) \text{-Hld}}, \  y =  \int f(\hat{\mathbb{X}}) \D \mathbb{X} \right \}\\
& = \inf \left  \{\tilde{I}^{\#} (\tilde{v}) \middle |  \tilde{v} \in \mathcal{H},\ (u, v) = \mathbb{K} (\tilde{v}), \  y = \int f (\hat{\mathbb{L}}(u,v)) \D \mathbb{L}(u,v) \right \},
\end{align*}
where $\tilde{I}^{\#\#}$ is defined in Theorem $\ref{ldp}$. 
\end{theorem}
\begin{proof}
By Theorems $\ref{reconst}$ and \ref{ldp} together with the contraction principle, we have the claim. 
\end{proof}

\subsection{Rough differential equation driven by an $(\alpha,\beta)$ rough path integral and the Freidlin--Wentzell LDP}
We now discuss the following type of rough differential equation (RDE) (in Lyons' sense; see Section~8.8 of \cite{FH}, for example):
\begin{equation}\label{RDE}
 \bar{S}_t =  \int_0^t \bar{\sigma}(\bar{S}_u) \D Y_u,
\end{equation}
where $\bar{S}_t=S_t-S_0$, $\bar{\sigma}(s)= \sigma(S_0 + s)$ and
\begin{equation}\label{RDEY}
     \quad Y =  \int f(\hat{\mathbb{X}}) \D \mathbb{X}  \in \Omega_{\alpha \text{-Hld}}([0,T], \mathbb{R}^d), \ \ \mathbb{X} \in \Omega_{(\alpha,\beta)\text{-Hld}}.
\end{equation}
\begin{thm} \label{thm:33}
Let $\sigma \in C^{3}_b$.
\begin{enumerate}
    \item[(i)]
 RDE $(\ref{RDE})$ driven by \eqref{RDEY} has a unique solution $\Phi(Y) = (Y,\bar{S})$, where
\begin{equation*}
\Phi : \Omega_{\alpha \text{-Hld}} ([0,T],\mathbb{R}^d) \times \mathbb{R} \rightarrow \Omega_{\alpha \text{-Hld}} ([0,T],\mathbb{R}^{d+1})
\end{equation*}
is the solution map of \eqref{RDE} that is locally Lipschitz continuous with respect to $d_{\alpha}$.
\item[(ii)]
The  first level of the last component $\bar{S}$ of the solution to RDE ($\ref{RDE}$) for \eqref{RDEY} with
$\mathbb{X}=\mathbb{X}(\omega)$ defined in Proposition~\ref{lift} gives the solution
$S(\omega) = S_0 + \bar{S}$ to the It\^o SDE ($\ref{rvm}$).
\end{enumerate}
\end{thm}
\begin{proof}
 (i) is a standard result from rough path theory; see e.g., Theorem~1 in \cite{Le} or Chapter~8 in \cite{FH}.
(ii) follows from Proposition~\ref{lift};
see Chapter~9 in \cite{FH}.
\end{proof}

\begin{thm}\label{thm:35}
Let $\sigma \in C^{3}_b$ and
 $\bar{S}^{\epsilon} :=  \Phi(Y^{\epsilon})$, where $\Phi$ is the solution map of Theorem~\ref{thm:33}. Then the sequence of the processes $\{\bar{S}^{\epsilon} \}_{\epsilon \geqq 0}$ satisfies the LDP on $\Omega_{\alpha\text{-Hld}}$ with speed $\epsilon^{-1}$ with good rate function 
\begin{align*}
\tilde{I} (\bar{s})&  
:= \inf \left \{ \tilde{I}^{\#\#\#} (Y) \middle|\  Y \in \Omega_{\alpha \text{-Hld}}, \  \bar{s} =  \Phi (Y) \right \}\\
& = \inf \left \{ \tilde{I}^{\#}(\tilde{v}) \middle| \tilde{v} \in \mathcal{H}, \ (u, v) = \mathbb{K} (\tilde{v}), \
\bar{s} = \int \bar{\sigma} (\bar{s}) f(\hat{\mathbb{L}}(u,v)) \D \mathbb{L}(u,v) \right \}.
\end{align*}
\end{thm}
\begin{proof}
Because the solution map $\Phi$ is continuous,  Theorem~\ref{thm:33} and the contraction theorem imply the claim.
\end{proof}
\subsection{Short-time asymptotics}
We consider the case $\kappa = \kappa_H$ (see (\ref{RLk})). By the scaling property of the Riemann--Liouville fractional Brownian motion $\hat{X}$ and the standard Brownian motion $X$, we have
\begin{equation*}
    \hat{X}_{\epsilon t} \sim \epsilon^H \hat{X}_t, \ \ 
    X_{\epsilon t} \sim  \epsilon^{1/2}X_t.
\end{equation*}
This implies
\begin{equation*}
\tilde{Y}^{\epsilon}_t := \epsilon^{H-1/2}    \int_0^{\epsilon t}  f(\hat{X}_u) \D X_u
\sim \int_0^t f(\hat{X}^{\epsilon}_u) \D X^{\epsilon}_u,
\end{equation*}
where $(\hat{X}^\epsilon,X^\epsilon) = \epsilon^H (\hat{X},X)$, of which the rough path lift is $\mathbb{X}^\epsilon$ of Theorem~\ref{ldp}.
Letting 
\begin{equation*}
    \tilde{S}^\epsilon_t = \frac{S_{\epsilon t} - S_0}{\epsilon^{1/2-H}}, \ \ \tilde{\sigma}^\epsilon(s) = \sigma(S_0 + \epsilon^{1/2-H}s), 
\end{equation*}
we have 
\begin{equation*}
    \tilde{S}^\epsilon_t = \int_0^t \tilde{\sigma}^\epsilon(\tilde{S}^\epsilon_u)\D \tilde{Y}^\epsilon_u,
\end{equation*}
and we can derive an LDP for $\tilde{S}^\epsilon$ by an extended contraction principle~\cite{Pu}.
\begin{thm}\label{thm:36}
Let $\sigma \in C^3_b$. Then 
$\{ \tilde{S}^{\epsilon} \}_{0 < \epsilon \leqq 1}$ satisfies the LDP 
on $\Omega_{\alpha\text{-Hld}}$ as $\epsilon \to 0$ with speed $\epsilon^{-2H}$ with good rate function
\begin{align*}
 \tilde{J} (\tilde{s}) 
 := \inf \left \{ \tilde{I}^{\#}(\tilde{v}) \middle|  \tilde{v} \in \mathcal{H}, \ (u, v) = \mathbb{K} (\tilde{v}), \ \tilde{s} = \sigma (S_0) \int  f(\hat{\mathbb{L}}(u,v)) \D \mathbb{L}(u,v) \right \}.
\end{align*}
\end{thm}
\begin{proof}
Denote by $\Phi_{\epsilon}$ the solution map of the RDE~\eqref{RDE} with $\bar{\sigma} = \tilde{\sigma}^\epsilon$.
We are going to show that $\Phi_\epsilon$ is locally equicontinuous.
Because for all $i \in \mathbb{Z}_+$,
$$||\nabla^i \tilde{\sigma}^{\epsilon}||_{\infty} \leqq (1+ \epsilon)^i || \nabla^i \sigma ||_{\infty} \leqq 2^i || \nabla^i \sigma ||_{\infty},$$ the local Lipschitz constants of $\Phi_{\epsilon}$ can be taken uniformly in $\epsilon$ by Theorem~4 in \cite{Le}. Therefore $\Phi_{\epsilon}$ is 
equicontinuous on bounded sets, and we conclude $\Phi_{\epsilon}(Y_{\epsilon}) \to \Phi_0 (Y)$
for any converging sequence $Y_\epsilon \to Y$ for any $Y$ with $\tilde{I}^{\#\#\#}(Y) < \infty$.
Then by Theorem~\ref{thm:34} and an extended contraction principle \cite{Pu}[\text{Theorem~2.1}],  we have the desired results.
\end{proof}

\begin{remark}
By  the usual argument, adding a drift term to the above RDE is straightforward. The result then generalizes  the existing LDP  for  the rough Bergomi model: 
\[
\D \log S_t = - \frac{1}{2} f^2(\hat{X}_t) \D t + f(\hat{X}_t) \D X_t
\]
in \cite{FoZh,BaFrGaMaSt, Gu, jacquier, GeJaPaStWa}.
To deal with the mixed rough Bergomi model
 \cite{BoDeGo} or the two-factor fractional volatility model \cite{FuKi},
 we need an extension with higher dimensional $\kappa$ and $W$ that is also straightforward.
\end{remark}

An LDP for the marginal distribution $\tilde{S}^\epsilon_1$ follows from the contraction principle, and the corresponding one-dimensional rate function extends the one obtained by \cite{FoZh} as follows.

\begin{theorem} \label{FZ}
Assume $\sigma \in C^3_b$ and
$|\rho|  < 1$. Then $t^{H-1/2} \bar{S}_t$ satisfies the LDP as $t\to 0$ with speed $t^{-2H}$ with good rate function
\[
\tilde{J}^{\#} (z) 
:= \inf_{g \in L^2([0,1])} 
\left [ \frac{1}{2} \int_0^1 |g_r|^2 \D r + \frac{ \left \{z - \rho \sigma(S_0) \int_0^1 f \left (K_Hg(r),0 \right) g_r \D r \right\}^2 }{2 (1 - \rho^2) \sigma(S_0)^2 \int_0^1 f\left(K_Hg(r) ,0 \right)^2 \D r } \right],
\]
where $K_Hg(t) = \int_0^t \kappa_H (t-r)g_r\D r$.
\end{theorem}
\begin{proof}
See Appendix \ref{C}.
\end{proof}
A short-time asymptotic formula of the implied volatility (regarding $S$ as a price or a log-price process) then follows from Theorem~\ref{FZ} as in \cite{FoZh}.
From the rate function of Theorem~\ref{FZ}, we observe that the effect of the function $\sigma$ to the short-time asymptotics is only through the constant $\sigma(S_0)$. In particular, the local volatility function $\sigma$ does not add any flexibility to the asymptotic shape of the implied volatility surface.

\section{Proofs of main theorems}
\subsection { Proof of Theorem \ref{reconst}}\label{pr1}
\begin{proof}
By a localizing argument, 
we can assume without loss of generality that the derivatives of $f$ are bounded.
For brevity, let $K  := ||f||_{C^{n+2}_b}$ and $M := |||\mathbb{X}|||_{(\alpha,\beta)}$. Let
\[
J^{(1)}_{st} = J^{(1)}(\mathbb{X})_{st} := \sum_{i \in I}  \partial^i f(\hat{x}_s) X^{(i)}_{st},
\quad J^{(2)}_{st} = J^{(2)}(\mathbb{X})_{st}  := \sum_{ (j,k) \in J} \partial^j f(\hat{x}_s) \partial^k f(\hat{x}_s)  \mathbf{X}^{(jk)}_{st},
\]
where $\hat{x}_s := \hat{X}_{0s}$.
Below, we follow the standard argument of rough path theory with Chen's identity replaced by our modified version \eqref{chen1}, \eqref{chen2}.
\begin{description}

\item[(Claim~1)]
 The first level of the $(\alpha, \beta)$ rough path integral $Y^{(1)}_{st}$ is well-defined and has the following inequality: 
\begin{equation}\label{step1y}
|Y^{(1)}_{st}|  \leqq K C_1 |t-s|^{\alpha},
\end{equation}
where
\[
C_1 := (n+1)^{2e} (1 + M)^{n+2} (1+T)^{(n+1)\beta} \left \{1 + 2^{(n+1)\beta + \alpha} \zeta((n+1)\beta + \alpha) \right\},
\]
and $\zeta(r) := \sum_{p=1}^{\infty} \frac{1}{p^r}$.
\begin{proof}
By Taylor expansion, we have
\begin{align}\label{tay1}
\sum_{i \in I} \partial^i f(\hat{x}_u) X^{(i)}_{ut} 
& =  \sum_{i \in I} \left \{   \sum_{|p| \leqq n - |i|} \frac{1}{ p!} \partial^{i+p} f(\hat{x}_{s}) \left (\hat{X}_{su}  \right)^p X^{(i)}_{ut} 
+ R_i X^{(i)}_{ut}  \right \} \notag \\
& = \sum_{i \in I}  \partial^i f(\hat{x}_s) \left \{ \sum_{p \leqq i}
\frac{1}{(i-p)!}  \left (\hat{X}_{su} \right)^{i-p} X^{(p)}_{ut} \right \} + \sum_{i \in I} R_i X^{(i)}_{ut},
\end{align}
where
\begin{align}\label{r}
   & R_i = R(\mathbb{X})_i \notag \\
&= \sum_{|p| = n+1-|i|}  \left( \int_0^1 \frac{(1-\theta)^{n+1-|i|} (n+1-|i|) }{p!} \partial^{p} f(\hat{x}_{s} + \theta \hat{X}_{su}) \D \theta \right ) ( \hat{X}_{su})^{p}
\end{align}
By the modified Chen's relation ($\ref{chen1}$) and $(\ref{tay1})$, for any $s \leqq u \leqq t$, 
\begin{align}\label{step1main}
 & J^{(1)}_{su} + J^{(1)}_{ut} - J^{(1)}_{st} \notag \\
 &= \sum_{i \in I}  \partial^i f(\hat{x}_s) \left (X^{(i)}_{su} - X^{(i)}_{st}\right) + \sum_{i \in I} \partial^i f(\hat{x}_u)  X^{(i)}_{ut} \notag \\
& = - \sum_{i \in I}  \partial^i f(\hat{x}_s) \left \{ \sum_{p\leqq i} \frac{1}{(i-p)!}  \left (\hat{X}_{su} \right)^{i-p} X^{(p)}_{ut} \right \} 
+ \sum_{i \in I} \partial^i f(\hat{x}_u) X^{(i)}_{ut} \notag\\
& =  \sum_{i \in I} R_i X^{(i)}_{ut} .
\end{align}
Because for all $i \in I$,
\begin{equation*}
\left | R_i X^{(i)}_{ut}   \right | 
 \leqq K \sum_{|p| = n+1-|i|} \left |  (\hat{X}_{su} )^{p} X^{(i)}_{ut} \right| 
 \leqq K (n+1)^e (1+M)^{n+2} |t-s|^{(n+1)\beta + \alpha}, 
\end{equation*}
we have
\begin{equation*}
\left |J^{(1)}_{su} + J^{(1)}_{ut} - J^{(1)}_{st} \right | \leqq K (n+1)^{2e}  (1 + M)^{n+2} |t-s|^{(n+1)\beta + \alpha}.
\end{equation*}
For any partition $\mathcal{P} = \{ s = t_0 < t_1 <...<t_N = t \} $, let $J^{(1)}_{st}(\mathcal{P}):= \sum_{p=1}^{N} J^{(1)}_{t_{p-1} t_{p}}$.  
By Lemma~$\ref{itten}$, there exists $p \in \{1,...,N \}$ such that 
\begin{align}\label{p}
     |t_{p+1} -  t_{p+1}| \leqq  \frac{2}{N-1} |t-s|. 
\end{align}
Then we have
\begin{align*}
 & \left |J^{(1)}_{st} \left (\mathcal{P} \right) - J^{(1)}_{st} \left (\mathcal{P} \backslash \{t_{p}\} \right) \right| \\
& = \left |J^{(1)}_{t_{p-1} t_{p}} + J^{(1)}_{t_{p} t_{p+1}} - J^{(1)}_{t_{p-1} t_{p+1}}  \right |\\
& \leqq K (n+1)^{2e}  (1 + M)^{n+2} |t_{p+1}-t_{p-1}|^{(n+1)\beta + \alpha}\\
& \leqq K (n+1)^{2e}  (1 + M)^{n+2} \left (\frac{2}{N-1} \right)^{(n+1)\beta+\alpha} |t-s|^{(n+1)\beta + \alpha},
\end{align*}
and this implies (note that $(n+1)\beta + \alpha > 1$)
\begin{align}\label{step1j}
 & \left |J^{(1)}_{st}(\mathcal{P}) - J^{(1)}_{st} \right| \notag\\
 &\leqq  K (n+1)^{2e}  (1 + M)^{n+2} 2^{(n+1)\beta +\alpha}  \zeta \left((n+1)\beta +\alpha \right) |t-s|^{(n+1)\beta + \alpha}.
\end{align}

\textbf{(Claim~1a)} $\{J^{(1)}_{st}(\mathcal{P})\}_{\mathcal{P}}$ is a Cauchy sequence with $|\mathcal{P}| \searrow 0$. 

Let $\mathcal{Q}$ be any subdivision of $\mathcal{P}$: $\mathcal{Q} = \{s = \tau_0 < \tau_1 <...<\tau_L = t , \} , L >N$. 
Consider the subsequence $\{ \tau_{l_0}< \tau_{l_1}< ...<\tau_{l_N} \}$ with $\tau_{l_p} = t_p$, and let $\mathcal{Q}_p := \mathcal{Q} \cap [t_{p-1},t_p]$. Then $\mathcal{Q}_p$ is a partition of $[t_{p-1},t_p]$. 
By using (\ref{step1j}), we have that 
\begin{align*}
    & \left | J^{(1)}_{st} (\mathcal{Q}) - J^{(1)}_{st} ( \mathcal{P})  \right| \\
    & \leqq \sum_{p=1}^N \left | J^{(1)}_{st} (\mathcal{Q}_p) - J^{(1)}_{t_{p-1} t_p }   \right|\\
    & \leqq K (n+1)^{2e}  (1 + M)^{n+2} 2^{(n+1)\beta +\alpha}  \zeta \left((n+1)\beta +\alpha \right) 
    \sum_{p=1}^N |t_p - t_{p-1}|^{(n+1)\beta + \alpha}\\
    & \leqq K (n+1)^{2e}  (1 + M)^{n+2} 2^{(n+1)\beta +\alpha}  \zeta \left((n+1)\beta +\alpha \right) T \left ( \sup_{t-s \leqq |\mathcal{P}|} |t-s|^{(n+1)\beta + \alpha -1} \right).
\end{align*}
Hence for any partition $\mathcal{P}, \mathcal{P}'$ with $| \mathcal{P} | \vee |\mathcal{P}'| \leqq \delta$, we have that 
\begin{align*}
    & \left | J^{(1)}_{st} (\mathcal{P}) - J^{(1)}_{st} ( \mathcal{P}')  \right| \\
    & \leqq \left | J^{(1)}_{st} (\mathcal{P}) - J^{(1)}_{st} (\mathcal{P} \cup \mathcal{P}')  \right| 
    + \left | J^{(1)}_{st} (\mathcal{P} \cup \mathcal{P}') - J^{(1)}_{st} ( \mathcal{P}')  \right| \notag \\
    & \leqq K (n+1)^{2e}  (1 + M)^{n+2} 2^{(n+1)\beta +\alpha+1}  \zeta \left((n+1)\beta +\alpha \right) T 
    \left ( \sup_{t-s \leqq \delta } |t-s|^{(n+1)\beta + \alpha -1} \right),
\end{align*}
and because $(n+1)\beta + \alpha > 1$, we conclude that $\{J^{(1)}_{st}(\mathcal{P})\}_{\mathcal{P}}$ is a Cauchy sequence.

Therefore, $Y^{(1)}_{st}$ is well-defined. Furthermore, by $(\ref{step1j})$, we have 
\begin{equation*}
|Y^{(1)}_{st}|  \leqq  |J^{(1)}_{st}| + |Y^{(1)}_{st} - J^{(1)}_{st}| \leqq K C_1 |t-s|^{\alpha}.
\end{equation*}
Thus we have proved the statement of Claim~1. 
\end{proof}

\item[(Claim~2)]
Let $m := \max_{(j,k) \in J} |j+k|$. Then the second level of the $(\alpha, \beta)$ rough path integral $Y^{(2)}_{st}$ is well-defined and has the following inequality:
\begin{equation*}
|Y^{(2)}_{st}| \leqq  K^2  C_2  |t-s|^{2\alpha},
\end{equation*}
where 
\[
C_2 := (1+m)^{2e} M (1+T)^{m \beta} + \left ( \tilde{C}_2 + 2  C^2_1 T^{(n  - m)\beta }  \right )  
2^{(m+1)\beta+2 \alpha} \zeta \left ((m+1)\beta + 2 \alpha \right ),
\]
and
\[
\tilde{C}_2 := 2 (1+n+m)^{4e} (1+M)^{m+3}  (1+T)^{(2n-m-1)\beta}  .
\]
In particular, we have $\int f(\hat{\mathbb{X}}) d\mathbb{X} \in \Omega_{\alpha \text{-Hld}}$. 

\begin{proof}
By the modified Chen's relation ($\ref{chen2}$), for all $s \leqq u \leqq t$, 
\begin{align*}
& J^{(2)}_{su} + J^{(2)}_{ut} + J^{(1)}_{su} \otimes J^{(1)}_{ut} - J^{(2)}_{st}  \\
& = J^{(1)}_{su} \otimes J^{(1)}_{ut} + \sum_{ (j,k) \in J} \left [ \partial^j f(\hat{x}_s) \partial^k f(\hat{x}_s) \left (\mathbf{X}^{(jk)}_{su} - \mathbf{X}^{(jk)}_{st} \right) +  \partial^j f(\hat{x}_u) \partial^k f(\hat{x}_u) \mathbf{X}^{(jk)}_{ut} \right]  \\
& = : S_1 + S_2,
\end{align*}
where
\begin{align*}
& S_1 = S_1(\mathbb{X}) \\
&:=  J^{(1)}_{su} \otimes J^{(1)}_{ut} 
 - \!\! \sum_{(j,k) \in J} \partial^j f(\hat{x}_s) \partial^k f(\hat{x}_s) 
 \left ( \sum_{q\leqq k} \frac {1}{(k-q)!} \left (\hat{X}_{su} \right )^{k-q} X^{(j)}_{su} \otimes X^{(q)}_{ut} \right)
\end{align*}
and 
\begin{align*}
& S_2 = S_2(\mathbb{X}) \\
& := \! \!\! \sum_{(j,k) \in J}  \partial^j f(\hat{x}_u)  \partial^k f(\hat{x}_u) \mathbf{X}^{(jk)}_{ut}\\
& \quad- \!\! \sum_{(j,k) \in J} \partial^j f(\hat{x}_s) \partial^k f(\hat{x}_s)\left(  \sum_{p \leqq j} \sum_{q \leqq k} \frac{1}{(j-p)! (k-q)!}  \left (\hat{X}_{su} \right)^{j+k-p-q} \mathbf{X}^{(pq)}_{ut} \right ).
\end{align*}
Note that 
\begin{align*}
 J^{(1)}_{su} \otimes J^{(1)}_{ut} 
 &= \left (\sum_{j \in I}  \partial^j f(\hat{x}_s) X^{(j)}_{su} \right ) 
\otimes \left (\sum_{k \in I} \partial^k f(\hat{x}_u) X^{(k)}_{ut} \right)\\
&= \sum_{ j \in I} \sum_{k \in I} \partial^j f(\hat{x}_s) \partial^k f(\hat{x}_u) X^{(j)}_{su} \otimes X^{(k)}_{ut}\\
& = \sum_{(j,k) \in J}  \partial^j f(\hat{x}_s) \partial^k f(\hat{x}_u) X^{(j)}_{su} \otimes X^{(k)}_{ut} 
+ \!\!\! \!\! \sum_{(j,k) \in I \times I \backslash J}  \partial^j f(\hat{x}_s) \partial^k f(\hat{x}_u) X^{(j)}_{su} \otimes X^{(k)}_{ut}.
\end{align*}
By Taylor expansion, we have
\begin{align*}
&  \sum_{(j,k) \in J}  \partial^j f(\hat{x}_s) \partial^k f(\hat{x}_u) X^{(j)}_{su} \otimes X^{(k)}_{ut} \\
& =\!\!\! \sum_{(j,k) \in J} \left \{  \sum_{|q| \leqq m - |j+k|} \frac{1}{q!}  \partial^j f(\hat{x}_s) \partial^{k+q} f(\hat{x}_s) (\hat{X}_{su})^{q} X^{(j)}_{su} \otimes X^{(k)}_{ut}  
+ \partial^j f(\hat{x}_s) R^{(1)}_{jk} X^{(j)}_{su} \otimes X^{(k)}_{ut} \right \} \\
& = \!\!\! \sum_{(j,k) \in J} \left \{  \partial^j f(\hat{x}_s) \partial^k f(\hat{x}_s) \left ( \sum_{q \leqq k} \frac {1}{(k-q)!} \left (\hat{X}_{su} \right )^{k-q} X^{(j)}_{su} \otimes X^{(q)}_{ut} \right) 
+ \partial^j f(\hat{x}_s) R^{(1)}_{jk} X^{(j)}_{su} \otimes X^{(k)}_{ut} \right \},
\end{align*}
where
\begin{align*}
  &R^{(1)}_{jk} = R^{(1)}_{jk}(\mathbb{X}) \\
  & := \sum_{|q| = m+1-|j+k|}  \left( \int_0^1 \frac{(1-\theta)^{m+1-|j+k|} (m+1-|j+k|)}{q !} \partial^{k+q} f(\hat{x}_s + \theta \hat{X}_{su} ) \D \theta \right)   (\hat{X}_{su})^{q},
\end{align*}
and so we obtain that
\begin{align}\label{s_1}
& S_1 = J^{(1)}_{su} \otimes J^{(1)}_{ut} 
 - \!\! \sum_{(j,k) \in J} \partial^j f(\hat{x}_s) \partial^k f(\hat{x}_s) \left ( \sum_{q \leqq k} \frac {1}{(k-q)!} \left (\hat{X}_{su} \right )^{k-q} X^{(j)}_{su} \otimes X^{(q)}_{ut} \right) \notag\\
 & = \sum_{(j,k) \in J} \partial^j f(\hat{x}_s)  R^{(1)}_{jk} X^{(j)}_{su} \otimes X^{(k)}_{ut} 
 + \sum_{(j,k) \in I \times I \backslash J}  \partial^j f(\hat{x}_s) \partial^k f(\hat{x}_u) X^{(j)}_{su} \otimes X^{(k)}_{ut}
\end{align}
and 
\begin{align}\label{s_12}
& \left | S _1 \right |\notag\\
& \leqq  \sum_{(j,k) \in J} \left | \partial^j f(\hat{x}_s) R^{(1)}_{jk} X^{(j)}_{su} \otimes X^{(k)}_{ut} \right | 
+ \sum_{(j,k) \in I \times I \backslash J} \left | \partial^j f(\hat{x}_s) \partial^k f(\hat{x}_u) X^{(j)}_{su} \otimes X^{(k)}_{ut} \right | \notag \\
& \leqq \sum_{(j,k) \in J} \sum_{|q|= m+1-|j+k|} K^2 |(\hat{X}_{su} )^q X^{(j)}_{su} \otimes X^{(k)}_{ut} | 
+\sum_{(j,k) \in I \times I \backslash J} K^2 |X^{(j)}_{su} \otimes X^{(k)}_{ut}| \notag \\  
& \leqq K^2 (1+m)^{3e} (1+M)^{m+3} |t-s|^{(m+1)\beta+2\alpha} \notag \\
& \quad + K^2 (1+n)^{2e} (1+M)^{2} (1+T)^{(2n-m-1)\beta} |t-s|^{(m+1)\beta+2\alpha} \notag \\
& \leqq 2 K^2 (1+n+m)^{3e} (1+M)^{m+3} (1+ T)^{(2n-m-1) \beta} | t-s |^{(m+1)\beta + 2 \alpha}.
\end{align}
Here we use $m \leqq n$ (because $(n+1)\beta + \alpha >1$, we have $(n+1)\beta + 2 \alpha >1$, and the definition of $m$ implies $m \leqq n$). 

On the other hand, one can show that 
\begin{equation}\label{s_2}
 S_2 
 = \sum_{(j,k)\in J}
 \left\{  \sum_{|p| \leqq m-|j+k|}
 \frac{1}{p!} \partial^{j+p} f(\hat{x}_s) R^{(3)}_{jkp} (\hat{X}_{su})^{p} \mathbf{X}^{(jk)}_{ut} 
 + \partial^{k} f(\hat{x}_u) R^{(2)}_{jk}  \mathbf{X}^{(jk)}_{ut} \right\}
\end{equation}
by using the Taylor expansion   
\begin{align*}
& \sum_{(j,k) \in J} \partial^j f(\hat{x}_u) \partial^k f(\hat{x}_u) \mathbf{X}^{(jk)}_{ut} \\
& = \sum_{(j,k)\in J} \left \{ \sum_{|p|\leqq m - |j+k|} \frac{1}{p!} \partial^{j+p} f(\hat{x}_s) (\hat{X}_{su})^p  + R^{(2)}_{jk} \right \} \partial^{k} f(\hat{x}_u) \mathbf{X}^{(jk)}_{ut}\\
& = \sum_{(j,k)\in J} \sum_{|p|\leqq m - |j+k|} \sum_{|q| \leqq m - |j+k+p|} \frac{1}{p! q!}\partial^{j+p} f(\hat{x}_s) \partial^{k+q} f(\hat{x}_s) (\hat{X}_{su})^{p+q} \mathbf{X}^{(jk)}_{ut}\\
& \quad + \sum_{(j,k)\in J} \sum_{|p|\leqq m - |j+k|} \frac{1}{p!}\partial^{j+p} f(\hat{x}_s) R^{(3)}_{jkp} (\hat{X}_{su})^{p} \mathbf{X}^{(jk)}_{ut}\\
& \quad + \sum_{(j,k)\in J} R^{(2)}_{jk} \partial^{k} f(\hat{x}_u)   \mathbf{X}^{(jk)}_{ut},
\end{align*}
where
\begin{align*}
  &R^{(2)}_{jk} = R^{(2)}_{jk} (\mathbb{X})\\
  &:= \sum_{|p| = m+1 -|j+k|} \left( \int_0^1 \frac{(1-\theta)^{m+1 -|j+k|} (m+1 -|j+k|)}{p !} 
  \partial^{j+p} f(\hat{x}_s + \theta \hat{X}_{su} ) \D \theta \right) (\hat{X}_{su})^{p},  
\end{align*}

\begin{align*}
  &R^{(3)}_{jkp} = R^{(3)}_{jkp} (\mathbb{X}) \\
  &:= \! \! \! \! \! \sum_{|q| = m+1 -|j+k+p|} \left( \int_0^1 \frac{(1-\theta)^{m+1 -|j+k+p|} (m+1 -|j+k+p|)}{ q !} 
  \partial^{k+q} f(\hat{x}_s + \theta \hat{X}_{su} ) \D \theta \right) (\hat{X}_{su})^{q}.
\end{align*}

Because for all $(j,k) \in J$ and $0\leqq |p|\leqq m-|j+k|$,
\begin{align*}
    \left| \partial^{j+p} f(\hat{x}_s) R^{(3)}_{jkp} (\hat{X}_{su})^{p} \mathbf{X}^{(jk)}_{ut} \right | 
&\leqq K^2 \sum_{|q| =m+1 - |j+k+p|} |(\hat{X}_{su})^{p+q} \mathbf{X}^{(jk)}_{ut}|\\
&\leqq K^2 (1+m)^e (1+M)^{m+2}  |t-s|^{(m+1)\beta +2 \alpha},
\end{align*}
and
\begin{align*}
 \left| R^{(2)}_{jk} \partial^{k} f(\hat{x}_u)  \mathbf{X}^{(jk)}_{ut} \right| 
&\leqq  K^2 \sum_{|p| =m+1 - |j+k|} |(\hat{X}_{su})^{p} \mathbf{X}^{(jk)}_{ut}|\\
&\leqq  K^2 (1+m)^e (1+M)^{m+2}  |t-s|^{(m+1)\beta +2 \alpha},   
\end{align*}
we have
\begin{align}\label{s_22}
& \left | S_2 \right | \notag \\
& \leqq \sum_{(j,k)\in J} \sum_{|p| \leqq m-|j+k|} | \partial^{j+p} f(\hat{x}_s) R^{(3)}_{jkp} (\hat{X}_{su})^{p} \mathbf{X}^{(jk)}_{ut} |
 + \sum_{(j,k)\in J} | \partial^{k} f(\hat{x}_u) R^{(2)}_{jk}  \mathbf{X}^{(jk)}_{ut} | \notag \\
 &\leqq K^2 (1+m)^{4e}(1+M)^{m+2} |t-s|^{(m+1)\beta +2 \alpha} + K^2 (1+m)^{3e} (1+M)^{m+1} |t-s|^{(m+1)\beta +2 \alpha} \notag\\
& \leqq2  K^2  (1+m)^{4e} (1+M)^{m+2}  |t-s|^{(m+1)\beta +2 \alpha}.
\end{align}
By $ (\ref{s_12})$ and $(\ref{s_22})$, we have
\begin{align*}
\left |J^{(2)}_{su} + J^{(2)}_{ut} + J^{(1)}_{su} \otimes J^{(1)}_{ut} - J^{(2)}_{st}  \right|  \leqq |S_1|+|S_2| \leqq K^2 \tilde{C}_2  |t-s|^{(m+1)\beta + 2\alpha},
\end{align*}
where $\tilde{C}_2 = 2 (1+n+m)^{4e} (1+M)^{m+3}  (1+T)^{(2n-m-1)\beta}  $.
Moreover, by  $(\ref{step1y})$ and $(\ref{step1j})$, we have 
\begin{align*}
& \left |Y^{(1)}_{su} \otimes Y^{(1)}_{ut} - J^{(1)}_{su} \otimes J^{(1)}_{ut} \right|  \\
&\leqq \left |Y^{(1)}_{su} \right|  \left|Y^{(1)}_{ut} - J^{(1)}_{ut} \right| 
+\left |Y^{(1)}_{su} - J^{(1)}_{su} \right| \left |J^{(1)}_{ut} \right| \leqq 2 K^2 C_1^2 |t-s|^{(n+1)\beta+2\alpha}.
\end{align*}
Let $J^{(2)}_{st}(\mathcal{P}):= \sum_{p=1}^{n}  Y^{(1)}_{t_{0} t_{p-1}} \otimes Y^{(1)}_{t_{p-1} t_p}  + J^{(2)}_{t_{p-1}t_p} $. 
By Lemma $\ref{itten}$, there exists $p \in \{1,...,N \}$ such that (\ref{p}) holds. 
Note that $m \leqq n$. Then, the above inequalities imply that 
\begin{align*}
& \left |J^{(2)}_{st}(\mathcal{P}) - J^{(2)}_{st}(\mathcal{P} \backslash \{t_p\}) \right| \\ 
&  \leqq \left  |J^{(2)}_{t_{p-1}t_p} + J^{(2)}_{t_pt_{p+1}} + Y^{(1)}_{t_{p-1}t_p} \otimes Y^{(1)}_{t_pt_{p+1}} - J^{(2)}_{t_{p-1}t_{p+1}} \right |\\
& \leqq \left |J^{(2)}_{t_{p-1}t_p} + J^{(2)}_{t_pt_{p+1}} + J^{(1)}_{t_{p-1}t_p} \otimes J^{(1)}_{t_pt_{p+1}} - J^{(2)}_{t_{p-1}t_{p+1}} \right|
+ \left |Y^{(1)}_{t_{p-1}t_p} \otimes Y^{(1)}_{t_pt_{p+1}} - J^{(1)}_{t_{p-1}t_p} \otimes J^{(1)}_{t_pt_{p+1}} \right|\\
& \leqq K^2 \tilde{C}_2 |t_{p+1} - t_{p-1}|^{(m+1)\beta + 2\alpha} + 2K^2 C_1^2  |t_{p+1} - t_{p-1}|^{(n+1)\beta+2\alpha} \\
& \leqq K^2 \left ( \tilde{C}_2 + 2 C^2_1 T^{(n  - m)\beta }  \right )  \left ( \frac{2}{N-1} \right)^{(m+1)\beta + 2\alpha} |t-s|^{(m+1)\beta +2 \alpha}.
\end{align*}
This implies that (note that $(m+1)\beta + 2 \alpha > 1$)
\begin{align}\label{it}
 |J^{(2)}_{st}(\mathcal{P}) - J^{(2)}_{st}| 
 \leqq K^2  C_2 |t-s|^{(m+1)\beta +2\alpha}.
\end{align}
This shows that  $\{J^{(2)}_{st}(\mathcal{P})\}_{\mathcal{P}}$ is a Cauchy sequence when $|\mathcal{P}| \searrow 0$ (one can adapt the argument of Claim~1a in the proof of Claim~1 by using (\ref{it}) instead of (\ref{step1j})).
Hence, $Y^{(2)}_{st}$ is well-defined. We also obtain that 
\begin{align*}
 |Y^{(2)}_{st}|  \leqq  |J^{(2)}_{st}| + |Y^{(2)}_{st} - J^{(2)}_{st}|  \leqq K^2  C_2  |t-s|^{2\alpha}.
\end{align*}
Next, we prove that $\int f(\mathbb{\hat{X}}) \D \mathbb{X}$ satisfies Chen's relation. Fix $\epsilon >0$ and $s<u<t$. By taking a partition $\mathcal{P} = \{s=t_0<t_1<...<t_N=t \}$ of $[s,t]$ small enough (which has the point $t_{\tilde{N}}=u$), we have
\begin{align*}
 & \left |Y^{(1)}_{st} -Y^{(1)}_{su} -Y^{(1)}_{ut} \right| \\
 & \leqq \left |Y^{(1)}_{st}- \sum_{p=1}^{N} J^{(1)}_{t_{p-1} t_p}\right| 
 + \left|Y^{(1)}_{su}- \sum_{p=1}^{\tilde{N}} J^{(1)}_{t_{p-1} t_p}\right| 
+ \left|Y^{(1)}_{ut}- \sum_{p=\tilde{N}+1}^{N} J^{(1)}_{t_{p-1} t_p}\right|\\
&\leqq 3 \epsilon
\end{align*}
and so the first level of $\int f(\mathbb{\hat{X}}) \D \mathbb{X}$ satisfies Chen's relation. 
Note that this result implies that 
\[
\sum_{q=1}^{N} Y^{(1)}_{t_0t_{q-1}} \otimes Y^{(1)}_{t_{q-1} t_q} 
= \sum_{0<p <q \leqq N} Y^{(1)}_{t_{p-1} t_p} \otimes Y^{(1)}_{t_{q-1} t_q}.
\]
Note also that  
\begin{align*}
& Y^{(1)}_{su} \otimes Y^{(1)}_{ut} \\
&= \left( \sum_{p=1}^{\tilde{N}} Y^{(1)}_{t_{p-1}t_p} \right) \otimes \left( \sum_{q=\tilde{N}+1}^N Y^{(1)}_{t_{q-1}t_q} \right )\\
& = \sum_{0<p<q \leqq N} Y^{(1)}_{t_{p-1}t_p} \otimes Y^{(1)}_{t_{q-1}t_q} 
- \sum_{0<p<q \leqq \tilde{N}} Y^{(1)}_{t_{p-1}t_p}\otimes Y^{(1)}_{t_{q-1}t_q}
- \sum_{\tilde{N}<p<q \leqq N} Y^{(1)}_{t_{p-1}t_p}\otimes Y^{(1)}_{t_{q-1}t_q}\\
& = \sum_{p=1}^N Y^{(1)}_{t_0 t_{p-1}} \otimes Y^{(1)}_{t_{p-1} t_{p}} 
-\sum_{p=1}^{\tilde{N}}  Y^{(1)}_{t_0 t_{p-1}} \otimes Y^{(1)}_{t_{p-1} t_{p}}
-\sum_{p=\tilde{N}+1}^N  Y^{(1)}_{t_{\tilde{N}} t_{p-1}} \otimes Y^{(1)}_{t_{p-1} t_{p}},
\end{align*}
and so we have
\begin{align*}
&  \left| Y^{(2)}_{st} - Y^{(2)}_{su} - Y^{(2)}_{ut} - Y^{(1)}_{su} \otimes Y^{(1)}_{ut} \right|\\
&  \leqq \left|Y^{(2)}_{st} - \mathcal{S}_{st}  \right| + \left|Y^{(2)}_{su} - \mathcal{S}_{su}  \right| + \left|Y^{(2)}_{ut} - \mathcal{S}_{ut}  \right|  \leqq 3 \epsilon,
\end{align*}
where 
\begin{equation*}
    \begin{split}
       &\mathcal{S}_{st} := \sum_{p=1}^{N} \left ( Y^{(1)}_{t_{0} t_{p-1}} \otimes Y^{(1)}_{t_{p-1} t_{p}}  + \sum_{ (j,k) \in J} \!\! \partial^j f(\hat{x}_{t_{p-1}}) \partial^k f(\hat{x}_{t_{p-1}}) \mathbf{X}^{(jk)}_{t_{p-1}t_{p}} \right ), \\
       &\mathcal{S}_{su} := \sum_{p=1}^{\tilde{N}} \left ( Y^{(1)}_{t_{0} t_{p-1}} \otimes Y^{(1)}_{t_{p-1} t_{p}}  + \sum_{ (j,k) \in J} \!\! \partial^j f(\hat{x}_{t_{p-1}}) \partial^k f(\hat{x}_{t_{p-1}}) \mathbf{X}^{(jk)}_{t_{p-1}t_{p}} \right ), \\
       &\mathcal{S}_{ut} := \sum_{p=\tilde{N}+1}^{N} \left ( Y^{(1)}_{t_{M} t_{p-1}} \otimes Y^{(1)}_{t_{p-1} t_{p}}  + \sum_{ (j,k) \in J} \!\! \partial^j f(\hat{x}_{t_{p-1}}) \partial^k f(\hat{x}_{t_{p-1}}) \mathbf{X}^{(jk)}_{t_{p-1}t_{p}} \right ).
    \end{split}
\end{equation*}
Therefore, the second level of $\int f(\mathbb{\hat{X}}) \D \mathbb{X}$ also satisfies Chen's relation. 
The above argument proves statement (i) of Theorem~$\ref{reconst}$.
\end{proof}

\item[(Claim~3)]
Suppose that there exist $M>0$ and $\epsilon >0$ such that 
\begin{equation*}
|\hat{V}_{st}| \vee |\hat{W}_{st}|  \leqq  M |t-s|^{\beta}, \quad 
|V^{(i)}_{st}| \vee | W^{(i)}_{st}|  \leqq M |t-s|^{|i|\beta + \alpha},
\end{equation*}
\begin{equation*}
|\mathbf{V}^{(jk)}_{st}| \vee |\mathbf{W}^{(jk)}_{st}|  \leqq  M |t-s|^{ |j+k| \beta + 2 \alpha},\quad 
|\hat{V}_{st} - \hat{W}_{st}|  \leqq  \epsilon |t-s|^{\beta}, 
\end{equation*}
and
\begin{equation*}
|V^{(i)} _{st} - W^{(i)}_{st}|  \leqq \epsilon |t-s|^{|i|\beta + \alpha}, \ \  
|\mathbf{V}^{(jk)}_{st} - \mathbf{W}^{(jk)}_{st}|  \leqq  \epsilon |t-s|^{|j+k|\beta + 2\alpha}.
\end{equation*}
Then,  there exists $C_3>0$ such that
\begin{equation}\label{claim3}
\left | \left (\int f(\hat{\mathbb{V}}) \D {\mathbb{V}} \right )^{(1)}_{st} -  \left (\int f(\hat{\mathbb{W}}) \D {\mathbb{W}} \right)^{(1)}_{st} \right| \leqq K \epsilon C_3 |t-s|^{\alpha},
\end{equation}
where 
\begin{align*}
 C_3 :=  (1+n)^{2e+1} (1+T)^{(n+1)\beta}   \{ 1 + (3e+2) (1+M)^{n+2} 2^{(n+1)\beta + \alpha} \zeta((n+1)\beta +\alpha)  \}   .
\end{align*}
\begin{proof}
By the assumption and the mean value theorem, we have
\begin{align}\label{claim31}
& \left |J^{(1)}(\mathbb{V})_{st} - J^{(1)}(\mathbb{W})_{st} \right| \notag\\
& = \left |\sum_{i \in I} \partial^i f(\hat{v}_s) V^{(i)} _{st} 
- \sum_{i \in I}  \partial^i f(\hat{w}_s) W^{(i)} _{st} \right| \notag \\
& \leqq  \sum_{i \in I} \left \{  | \partial^i f(\hat{v}_s) - \partial^i f(\hat{w}_s)||V^{(i)}_{st}| 
+ |\partial^i f(\hat{w}_s)| | V^ {(i)} _{st} - W^{(i)}_{st}  | \right \} \notag \\
& \leqq K \epsilon (1+eM) (1+n)^e  (1+T)^{(n+1)\beta} |t-s|^{\alpha}.
\end{align}
By ($\ref{r}$), ($\ref{step1main}$), and the mean value theorem, for all $s \leqq u \leqq t$, 
\begin{align*}
& \left |J^{(1)}(\mathbb{V})_{su} + J^{(1)}(\mathbb{V})_{ut} - J^{(1)}(\mathbb{V})_{st} - \left \{ J^{(1)}(\mathbb{W})_{su} + J^{(1)}(\mathbb{W})_{ut} - J^{(1)}(\mathbb{W})_{st} \right \}  \right| \\
 &\leqq  \sum_{i \in I} \left | R_i (\mathbb{V}) V^{(i)}_{ut}  - R_i(\mathbb{W})  W^{(i)}_{ut} \right|\\
 &\leqq  \sum_{i \in I} \left | R_i (\mathbb{V}) - R_i(\mathbb{W}) | |V^{(i)}_{ut}| + | R_i(\mathbb{W}) |  |V^{(i)}_{ut}-  W^{(i)}_{ut} \right|\\
& \leqq (2e+1) K \epsilon  (1+n)^{2e+1} (1+T)^{\beta} (1+M)^{n+2} |t-s|^{(n+1)\beta + \alpha} \\
& \quad + K \epsilon  (1+n)^{2e} (1+M)^{n+1} |t-s|^{(n+1)\beta + \alpha}  \\
& \leqq (2e+2) K \epsilon  (1+n)^{2e+1}   (1+T)^{\beta} (1+M)^{n+2} |t-s|^{(n+1)\beta + \alpha}.
\end{align*}
By Lemma~$\ref{itten}$, there exists $p \in \{1,...,N \}$ such that (\ref{p}) holds. By the above inequality, we have that
\begin{align*}
 &\left  |J^{(1)}(\mathbb{V})_{st}(\mathcal{P}) - J^{(1)}(\mathbb{V})_{st}(\mathcal{P} \backslash \{ t_p \}) - \left \{ J^{(1)}(\mathbb{W})_{st}(\mathcal{P}) - J^{(1)}(\mathbb{W})_{st} \left (\mathcal{P} \backslash \{t_p\} \right) \right \} \right| \\
& = \bigg | J^{(1)}(\mathbb{V})_{t_{p-1}t_p} + J^{(1)}(\mathbb{V})_{t_pt_{p+1}} - J^{(1)}(\mathbb{V})_{t_{p-1}t_{p+1}} \\
& \hspace{1.5cm} - \left \{ J^{(1)}(\mathbb{W})_{t_{p-1}t_p} + J^{(1)}(\mathbb{W})_{t_pt_{p+1}} - J^{(1)}(\mathbb{W})_{t_{p-1}t_{p+1}} \right \} \bigg| \\
& \leqq (2e+2) K \epsilon  (1+n)^{2e+1}   (1+T)^{\beta} (1+M)^{n+2}  |t_{p+1} - t_{p-1} |^{(n+1)\beta + \alpha}\\
& \leqq (2e+2) K \epsilon  (1+n)^{2e+1}   (1+T)^{\beta} (1+M)^{n+2}  \left (\frac{2}{N-1} \right)^{(n+1)\beta + \alpha} |t-s|^{(n+1)\beta +\alpha}.
\end{align*}
This implies that (note that $(n+1)\beta + \alpha > 1$)
\begin{align}\label{claim32}
&  \left  |J^{(1)}(\mathbb{V})_{st}(\mathcal{P}) - J^{(1)}(\mathbb{V})_{st} - \{ J^{(1)}(\mathbb{W})_{st}(\mathcal{P}) - J^{(1)}(\mathbb{W})_{st} \} \right | \notag \\
& \leqq (2e+2) K \epsilon  (1+n)^{2e+1}   (1+T)^{\beta} (1+M)^{n+2}  2^{(n+1)\beta + \alpha} \zeta((n+1)\beta +\alpha) |t-s|^{(n+1)\beta +\alpha}.
\end{align}
Therefore, by $(\ref{claim31})$ and $(\ref{claim32})$, we conclude that 
\begin{align*}
&\left |J^{(1)}(\mathbb{V})_{st}(\mathcal{P}) - J^{(1)}(\mathbb{W})_{st}(\mathcal{P}) \right| \\
& \leqq  
\left |J^{(1)}(\mathbb{V})_{st} - J^{(1)}(\mathbb{W})_{st} \right |  + \left |J^{(1)}(\mathbb{V})_{st}(\mathcal{P}) - J^{(1)}(\mathbb{V})_{st} - \left \{  J^{(1)}(\mathbb{W})_{st}(\mathcal{P}) - J^{(1)}(\mathbb{W})_{st} \right \} \right|\\
& \leqq  K \epsilon (1+eM) (1+n)^e  (1+T)^{(n+1)\beta} |t-s|^{\alpha}\\
& \quad + (2e+2) K \epsilon  (1+n)^{2e+1}   (1+T)^{\beta} (1+M)^{n+2}  2^{(n+1)\beta + \alpha} \zeta((n+1)\beta +\alpha) |t-s|^{(n+1)\beta +\alpha} \\
& \leqq K \epsilon C_3 |t-s|^{\alpha}.
\end{align*}
Taking $|\mathcal{P}| \searrow 0$, we prove \eqref{claim3}.
\end{proof}

\item[(Claim~4)]
Suppose that there exist $M>0$ and $\epsilon >0$ such that 
\begin{equation*}
|\hat{V}_{st}| \vee |\hat{W}_{st}|  \leqq  M |t-s|^{\beta}, \quad 
|V^{(i)}_{st}| \vee | W^{(i)}_{st}|  \leqq M |t-s|^{|i|\beta + \alpha},
\end{equation*}
\begin{equation*}
|\mathbf{V}^{(jk)}_{st}| \vee |\mathbf{W}^{(jk)}_{st}|  \leqq  M |t-s|^{ |j+k| \beta + 2 \alpha},\quad 
|\hat{V}_{st} - \hat{W}_{st}|  \leqq  \epsilon |t-s|^{\beta}, 
\end{equation*}
and
\begin{equation*}
|V^{(i)} _{st} - W^{(i)}_{st}|  \leqq \epsilon |t-s|^{|i|\beta + \alpha}, \ \  
|\mathbf{V}^{(jk)}_{st} - \mathbf{W}^{(jk)}_{st}|  \leqq  \epsilon |t-s|^{|j+k|\beta + 2\alpha}.
\end{equation*}
Then 
\begin{equation}\label{claim4}
\left | \left (\int f(\hat{\mathbb{V}}) \D {\mathbb{V}} \right )^{(2)}_{st} - \left (\int f(\hat{\mathbb{W}}) \D {\mathbb{W}} \right)^{(2)}_{st} \right| \leqq K^2 \epsilon C_4  |t-s|^{2 \alpha},
\end{equation}
where
\[
C_4 := (1+m)^{2e} (1+2eM) (1+T)^{(m+1)\beta} + (1+T^{(n-m)\beta}) (\tilde{C}_4 + 4 C_1 C_3) 2^{(m+1)\beta + 2\alpha} \zeta((m+1)\beta + 2\alpha),
\]
\[
\tilde{C}_4 := (15e + 7)(1+n+m)^{3e} (1+M)^{m+3} (1+T)^{(2n-m)\beta}.
\]
In particular, the integration map is Lipschitz continuous.

\begin{proof}
The assumption and the mean value theorem imply that
\begin{align}\label{claim41}
& \left |J^{(2)}(\mathbb{V})_{st} - J^{(2)}(\mathbb{W})_{st} \right | \notag \\
&\leqq \sum_{(j,k)\in J} 
 \left | \partial^j f(\hat{v}_s) \partial^k f(\hat{v}_s) \mathbf{V}^{(jk)}_{st} - \partial^j f(\hat{w}_s) \partial^k f(\hat{w}_s) \mathbf{W}^{(jk)}_{st} \right | \notag \\
 &  \leqq  K^2 \epsilon (1+m)^{2e} (2 e M+1) (1+T)^{(m+1)\beta}  |t-s|^{ 2 \alpha}.
\end{align}
On the other hand, by $(\ref{s_1})$ and $(\ref{s_2})$, we can calculate
\begin{align*}
  & \left | S_1(\mathbb{V}) - S_1(\mathbb{W})  \right|  \\
  & \leqq \sum_{(j,k) \in J}
  | \partial^j f(\hat{v}_s)  R^{(1)}_{jk}(\mathbb{V}) V^{(j)}_{su} \otimes V^{(k)}_{ut} 
  -\partial^j f(\hat{w}_s)  R^{(1)}_{jk} (\mathbb{W}) W^{(j)}_{su} \otimes W^{(k)}_{ut}  |\\
  & \quad + \sum_{(j,k) \in I \times I \backslash J} 
  | \partial^j f(\hat{v}_s) \partial^k f(\hat{v}_u) V^{(j)}_{su} \otimes V^{(k)}_{ut}
  -\partial^j f(\hat{w}_s) \partial^k f(\hat{w}_u) W^{(j)}_{su} \otimes W^{(k)}_{ut}|\\
&\leqq K^2 \epsilon (1+m)^{3e} (1+M)^{m+3} (1+T)^{\beta} (5e+2) |t-s|^{(m+1)\beta+2\alpha} \\
& \quad + K^2 \epsilon (1+n)^{2e} (1+M)^2 (1+T)^{(2n-m)\beta} (2e+2)|t-s|^{(m+1)\beta+2\alpha} \\
&\leqq K^2 \epsilon (1+n+m)^{3e} (1+M)^{m+3}(1+T)^{(2n-m)\beta} (7e+4)  |t-s|^{(m+1)\beta + 2 \alpha},
\end{align*}
and 
\begin{align*}
 & \left | S_2(\mathbb{V}) - S_2(\mathbb{W})  \right|\\
 &\leqq \sum_{(j,k)\in J}\sum_{|p| \leqq m-|j+k|}
 \left|  \partial^{j+p} f(\hat{v}_s) R^{(3)}_{jkp} (\mathbb{V}) (\hat{V}_{su})^{p} \mathbf{V}^{(jk)}_{ut} 
 -  \partial^{j+p} f(\hat{w}_s) R^{(3)}_{jkp} (\mathbb{W}) (\hat{W}_{su})^{p} \mathbf{W}^{(jk)}_{ut} \right|\\
 & \quad + \sum_{(j,k)\in J}\sum_{|p| \leqq m-|j+k|} \left| \partial^{k} f(\hat{v}_u) R^{(2)}_{jk} (\mathbb{V}) \mathbf{V}^{(jk)}_{ut} 
 - \partial^{k} f(\hat{w}_u) R^{(2)}_{jk} (\mathbb{W}) \mathbf{W}^{(jk)}_{ut}\right|\\
 &\leqq K^2 \epsilon  (5e+2) (1+m)^{3e+1} (1+T)^{\beta} (1+M)^{m+2}|t-s|^{(m+1)\beta+2\alpha}\\
 & \quad + K^2 \epsilon  (3e+1) (1+m)^{3e+1} (1+M)^{m+2}|t-s|^{(m+1)\beta+2\alpha} \\
 &\leqq  K^2 \epsilon (8e +3)  (1+m)^{3e+1} (1+T)^{\beta} (1+M)^{m+2} |t-s|^{(m+1)\beta +2\alpha}.
\end{align*}
Therefore, we have
\begin{align*}
& \left | \Sigma(\mathbb{V})_{sut}  - \Sigma(\mathbb{W})_{sut} \right|\\
& \leqq \left | S_1(\mathbb{V}) - S_1(\mathbb{W}) \right | + \left | S_2(\mathbb{V}) - S_2(\mathbb{W}) \right | \leqq  K^2 \epsilon \tilde{C}_4  |t-s|^{(m+1)\beta +2\alpha }, 
\end{align*}
where
\[
\Sigma_{sut} (\mathbb{V}) := J^{(2)}(\mathbb{V})_{su}  + J^{(2)}(\mathbb{V})_{ut} 
+ J^{(1)}(\mathbb{V})_{su} \otimes J^{(1)}(\mathbb{V})_{ut} - J^{(2)}(\mathbb{V})_{st}, \quad s \leqq u \leqq t
\]
and 
\begin{align*}
 \tilde{C}_4 = (15e + 7)(1+n+m)^{3e} (1+M)^{m+3} (1+T)^{(2n-m)\beta} .
\end{align*}
Let 
\[
\Gamma(\mathbb{V})_{sut} := Y^{(1)}(\mathbb{V})_{su} \otimes Y^{(1)}(\mathbb{V})_{ut} 
- J^{(1)}(\mathbb{V})_{su} \otimes J^{(1)}(\mathbb{V})_{ut}, \quad s \leqq u \leqq t.
\]  
Then by $(\ref{step1y})$, $(\ref{step1j})$, $(\ref{claim3})$, and $(\ref{claim32})$, we have
\begin{align*}
 &  \left | \Gamma(\mathbb{V})_{sut}  -\Gamma(\mathbb{W})_{sut}  \right | \\
& \leqq \big|Y^{(1)}(\mathbb{V})_{su} \otimes (Y^{(1)}(\mathbb{V})_{ut} - J^{(1)}(\mathbb{V})_{ut}) 
- Y^{(1)}(\mathbb{W})_{su} \otimes (Y^{(1)}(\mathbb{W})_{ut} - J^{(1)}(\mathbb{W})_{ut}) \big|\\
& \quad +  \big|(Y^{(1)}(\mathbb{V})_{su} - J^{(1)}(\mathbb{V})_{su} ) \otimes J^{(1)}(\mathbb{V})_{ut}
- (Y^{(1)}(\mathbb{W})_{su} - J^{(1)}(\mathbb{W})_{su}) \otimes  J^{(1)}(\mathbb{W})_{ut} \big|\\
 & \leqq \left | Y^{(1)}(\mathbb{V})_{su} \right | \left | Y^{(1)}(\mathbb{V})_{ut} - J^{(1)}(\mathbb{V})_{ut} - Y^{(1)}(\mathbb{W})_{ut} + J^{(1)}(\mathbb{W})_{ut} \right | \\
& \quad + \left |Y^{(1)}(\mathbb{V})_{su} - Y^{(1)}(\mathbb{W})_{su} \right | \left |Y^{(1)}(\mathbb{W})_{ut} - J^{(1)}(\mathbb{W})_{ut} \right|\\
& \quad +\left | Y^{(1)}(\mathbb{V})_{ut} - J^{(1)}(\mathbb{V})_{ut} - Y^{(1)}(\mathbb{W})_{ut} + J^{(1)}(\mathbb{W})_{ut} \right | \left | J^{(1)}(\mathbb{V})_{ut} \right | \\
& \quad + \left | Y^{(1)}(\mathbb{W})_{ut} - J^{(1)}(\mathbb{W})_{ut} \right|  \left|J^{(1)}(\mathbb{V})_{ut} - J^{(1)}(\mathbb{W})_{ut} \right |\\
 & \leqq K^2 \epsilon 4 C_1 C_3 |t-s|^{(n+1)\beta + 2\alpha}.
\end{align*}
By Lemma~$\ref{itten}$, there exists $p \in \{1,...,N \}$ such that (\ref{p}) holds. Then we have
\begin{align*}
&  \left  | J^{(2)}(\mathbb{V})_{st}(\mathcal{P}) - J^{(2)}(\mathbb{V})_{st}(\mathcal{P} \backslash \{t_p \})  - \left \{ J^{(2)}(\mathbb{W})_{st}(\mathcal{P} ) - J^{(2)}(\mathbb{W})_{st}(\mathcal{P} \backslash \{t_p \}) \right \} \right | \\
& \leqq \left  | \Sigma(\mathbb{V})_{t_{p-1} t_p t_{p+1}}  -  \Sigma(\mathbb{W})_{t_{p-1} t_p t_{p+1} } \right| + \left | \Gamma(\mathbb{V})_{t_{p-1} t_p t_{p+1}} - \Gamma(\mathbb{V})_{t_{p-1} t_p t_{p+1}} \right |\\
& \leqq   K^2 \epsilon \tilde{C}_4 |t_{i+1}-t_{i-1}|^{(m+1)\beta +2\alpha}  + K^2 \epsilon 4 C_1 C_3  |t_{i+1}-t_{i-1}|^{(n+1)\beta + 2\alpha} \\
& \leqq  K^2 \epsilon (1+T^{(n-m)\beta}) (\tilde{C}_4 + 4 C_1 C_3) \left (\frac{2}{N-1} \right )^{(m+1)\beta + 2\alpha } |t-s|^{(m+1)\beta + 2\alpha}. 
\end{align*}
This implies that (note that $(m+1)\beta + 2\alpha > 1$)
\begin{align}\label{claim42}
& \left |J^{(2)}(\mathbb{V})_{st}(\mathcal{P}) - J^{(2)}(\mathbb{V})_{st} - \left \{ J^{(2)}(\mathbb{W})_{st}(\mathcal{P}) - J^{(2)}(\mathbb{W})_{st} \right \} \right | \notag \\
& \leqq K^2 \epsilon (1+T^{(n-m)\beta}) (\tilde{C}_4 + 4 C_1 C_3) 2^{(m+1)\beta + 2\alpha} \zeta((m+1)\beta + 2\alpha) |t-s|^{(m+1)\beta+2\alpha}.
\end{align}
Therefore, by $(\ref{claim41})$ and $(\ref{claim42})$ we conclude that
\begin{align*}
& \left |J^{(2)}(\mathbb{V})_{st}(\mathcal{P}) -  J^{(2)}(\mathbb{W})_{st}(\mathcal{P}) \right| \notag \\
 & \leqq \left|J^{(2)}(\mathbb{V})_{st} - J^{(2)}(\mathbb{W})_{st} \right|\\
& \quad +  \left |J^{(2)}(\mathbb{V})_{st}(\mathcal{P}) - J^{(2)}(\mathbb{V})_{st} - \{ J^{(2)}(\mathbb{W})_{st}(\mathcal{P}) - J^{(2)}(\mathbb{W})_{st} \} \right |\\
 & \leqq K^2 \epsilon C_4 |t-s|^{2\alpha}.
\end{align*}
Taking $|\mathcal{P}| \searrow 0$, we have \eqref{claim4}.

For any $\mathbb{V}, \mathbb{W} \in \mathcal{E}_M$, take $\epsilon := d_{(\alpha,\beta)}(\mathbb{V}, \mathbb{W}) $. Then we have
\begin{equation*}
|\hat{V}_{st}| \vee |\hat{W}_{st}|  \leqq  M |t-s|^{\beta}, \quad 
|V^{(i)}_{st}| \vee | W^{(i)}_{st}|  \leqq M |t-s|^{|i|\beta + \alpha},
\end{equation*}
\begin{equation*}
|\mathbf{V}^{(jk)}_{st}| \vee |\mathbf{W}^{(jk)}_{st}|  \leqq  M |t-s|^{ |j+k| \beta + 2 \alpha},\quad 
|\hat{V}_{st} - \hat{W}_{st}|  \leqq  \epsilon |t-s|^{\beta}, 
\end{equation*}
and
\begin{equation*}
|V^{(i)} _{st} - W^{(i)}_{st}|  \leqq \epsilon |t-s|^{|i|\beta + \alpha}, \ \  
|\mathbf{V}^{(jk)}_{st} - \mathbf{W}^{(jk)}_{st}|  \leqq  \epsilon |t-s|^{|j+k|\beta + 2\alpha}.
\end{equation*}
Therefore, by (\ref{claim3}) and (\ref{claim4}) we conclude that for all $ \mathbb{V}, \mathbb{W} \in \mathcal{E}_M$,
\begin{align*}
 d_{\alpha}  \left (\int f(\hat{\mathbb{V}}) \D \mathbb{V}, \int f(\hat{\mathbb{W}}) \D \mathbb{W} \right) 
  & \leqq K C_3  \epsilon +  K^2 C_4 \epsilon  \\
  & \leqq K (C_3 + K C_4) d_{(\alpha,\beta)} (\mathbb{V},\mathbb{W}),
\end{align*}
and this is the claim.
\end{proof}
\end{description}
Claims~1--4 complete the proof of Theorem~$\ref{reconst}$.
\end{proof}

\subsection{Proof of Proposition~\ref{lift}}\label{proof_lift}
We use the following lemmas.
\begin{lemma}[\cite{Nu} Proposition 1.1.2]\label{wiener}
\[
I_1 (g) I_p (g^{\otimes p}) = I_{p+1} (g^{\otimes (p+1)}) + p ||g||^2_{L^2} I_{p-1} (g^{\otimes (p-1)}),
\quad g \in L^2(\mathbb{R}_+), \ p \geqq 1.
\]
\end{lemma}
\begin{lemma}[\cite{inahama} Corollary 9.7]\label{hyper}
Let $Y$ belong to the $m$-th Wiener chaos and $p \geqq 2$. Then we have
\[
||Y||_{p} \leqq \sqrt{m+1}(p-1)^{m/2} || Y ||_{2}. 
\] 
\end{lemma}

\begin{proof}[Proof of Proposition~\ref{lift}]
(i) Because $\gamma < 1/2$, $\hat{X}$ is well-defined and one can prove that $\mathcal{K} W(t) = \int_0^t \kappa (t-r) \D W_r$. The modified Chen's relation follows from the binomial theorem as illustrated in the Introduction.
For the \hd property,
by Kolmogorov's continuity theorem (see Theorem 3.1 in \cite{FH}), it is sufficient to prove the following inequalities:  for $p \geqq 2, \  i\in I, \ (j,k) \in J$, and $ (s,t) \in \Delta_T$
\[
|| X^{(i)}_{st} ||_{p} 
\leqq C|t-s|^{|i| \zeta + 1/2}, \quad 
|| \mathbf{X}^{(jk)}_{st}||_{p} \leqq C |t-s|^{|j+k| \zeta + 1}.
\]
Fix $s < r < t$. Note that $\hat{X}^{(1)}_{sr} = \int_0^r \kappa_{sr} (u) \D W_u$. Then by using Lemma~\ref{wiener} repeatedly, we have that for all $m \in \mathbb{Z}_{+}$,
\[
    \left( \hat{X}^{(1)}_{sr} \right)^{2m} 
    = \sum_{l=0}^m \tilde{c}_{l,m} I_{2l} (\kappa_{sr}^{\otimes 2l} ) ||\kappa_{sr}||^{2m-2l}_{L^2}, 
\]
\[
     \left( \hat{X}^{(1)}_{sr} \right)^{2m+1} 
    = \sum_{l=0}^m c_{l,m} I_{2l+1} (\kappa_{sr}^{\otimes (2l +1)} ) ||\kappa_{sr}||^{2m-2l}_{L^2},
\]
where $c_{0,0} =1$,
\[
\tilde{c}_{l,m} = 
\begin{cases}
    c_{0,m-1} & l=0,\\
    c_{l-1,m-1} + (2l+1)c_{l,m-1} & l=1,...,m-1,\\
    1 & \text{otherwise},
\end{cases}
\]
and
\[
c_{l,m} = 
\begin{cases}
    \tilde{c}_{l,m} + 2(l+1) \tilde{c}_{l+1,m} & l=0,...,m-1,\\
    1 & \text{otherwise}.
\end{cases}
\]
Then the assumption $\gamma < 1/2 $ and Lemma~\ref{hyper} imply that for all $m \in \mathbb{Z}_+$, 
\begin{align*}
    & \left| \left| \int_s^t \left( \hat{X}^{(1)}_{sr} \right)^{2m} \left( \hat{X}^{(2)}_{sr} \right)^{i_2} \D X_r    \right| \right|_{p} \\
    & \leqq \sum_{l=0}^m \tilde{c}_{l,m} \left| \left| \int_s^t I_{2l} (\kappa_{sr}^{\otimes 2l}) ||\kappa_{sr}||^{2m-2l}_{L^2} |r-s|^{\zeta i_2} \D X_r    \right| \right|_{p}\\
    & \leqq \sum_{l=0}^m \tilde{c}_{l,m} p^{(2l+1)/2} \left| \left| \int_s^t I_{2l} (\kappa_{sr}^{\otimes 2l}) ||\kappa_{sr}||^{2m-2l}_{L^2} |r-s|^{\zeta i_2} \D X_r    \right| \right|_{2}\\
    & \leqq p^{(i_1+1)/2} \left( \sum_{l=0}^m \tilde{c}_{l,m} p^{l-m} \right) |t-s|^{|i|\zeta +1/2 - i_1/2(2\gamma -1)}\\
    & \leqq C p^{(i_1+1)/2} \left( \sum_{l=0}^m \tilde{c}_{l,m} p^{l-m} \right ) |t-s|^{|i|\zeta +1/2},
\end{align*}
and 
\begin{align*}
   & \left| \left| \int_s^t \left( \hat{X}^{(1)}_{sr} \right)^{2m+1} \left( \hat{X}^{(2)}_{sr} \right)^{i_2} \D X_r    \right| \right|_{p} 
    \\& \leqq \sum_{l=0}^m c_{l,m} \left| \left| \int_s^t I_{2l+1} (\kappa_{sr}^{\otimes (2l+1)}) ||\kappa_{sr}||^{2m-2l}_{L^2} |r-s|^{\zeta i_2} \D X_r    \right| \right|_{p}\\
    & \leqq \sum_{l=0}^m c_{l,m} p^{(2l+2)/2} \left| \left| \int_s^t I_{2l+1} (\kappa_{sr}^{\otimes (2l+1)}) ||\kappa_{sr}||^{2m-2l}_{L^2} |r-s|^{\zeta i_2} \D X_r    \right| \right|_{2}\\
    & \leqq p^{(i_1+1)/2} \left( \sum_{l=0}^m c_{l,m} p^{l-m} \right) |t-s|^{|i|\zeta +1/2 - i_1/2(2\gamma -1)}\\
    & \leqq C p^{(i_1+1)/2} \left( \sum_{l=0}^m c_{l,m} p^{l-m} \right ) |t-s|^{|i|\zeta +1/2}.
\end{align*}
Therefore, we conclude that for all $i = (i_1,i_2) \in \mathbb{Z}_+^2$, 
\begin{align}\label{Lp1}
  || X^{(i)}_{st} ||_{p} =  \left| \left| \int_s^t \left( \hat{X}^{(1)}_{sr} \right)^{i_1} \left( \hat{X}^{(2)}_{sr} \right)^{i_2} \D X_r    \right| \right|_{p} 
    \leqq C p^{ (i_1 +1)/2} |t-s|^{|i|\zeta + 1/2},
\end{align}
and this implies the claim. By the same argument, we have
\begin{equation}\label{Lp2}
|| \mathbf{X}^{(jk)}_{st}||_{p} \leqq C p^{(j_1+k_1+2)/2} |t-s|^{|j+k|\zeta + 1} , \quad (j,k) = ((j_1,j_2), (k_1,k_2)) \in \mathbb{Z}_+^2 \times \mathbb{Z}_+^2. 
\end{equation}

(ii) By (i) and Theorem~$\ref{reconst}$, for a.s.\ $\omega$, the limit
\[
\left( \int f(\hat{\mathbb{X}}) \D {\mathbb{X}} \right )_{st}^{(1)} 
= \lim_{N \to \infty} \sum_{q =1}^{N} \sum_{i \in I} \partial^i f(\hat{X}_{t_{q-1}})X^{(i)}_{t_{q-1} t_{q}}
\]
exists. Because 
\[
\int_{s}^{t} f(\hat{X}_r) \D X_r = \lim_{N \to \infty} \sum_{q =1}^{N}  f(\hat{X}_{t_{q-1}})X^{(0)}_{t_{q-1} t_{q}}
\]
 in the sense of the convergence in probability, 
 it is sufficient to prove that for all $i \in I \backslash \{ 0\}$,  
\[
\lim_{N \to \infty} \sum_{q =1}^{N} \partial^i f(\hat{X}_{t_{q-1}})X^{(i)}_{t_{q-1} t_{q}} = 0
\]
in probability. Fix $i \in I \backslash \{ 0\}$. 
We can assume $f \in C^{n+2}_b$ without loss of generality.
By the result (i), we have
\[
\EXP{ \left (X^{(i)}_{st} \right)^2} = C |t-s|^{2|i| \zeta +1} < \infty,
\]
and so taking $K := ||f||_{C_b^{n+2}}$, we conclude that
\begin{align*}
\EXP{ \left( \sum_{q=1}^{N} \partial^i f(\hat{x}_{t_{q-1}}) X^{(i)}_{t_{q-1} , t_q} \right)^2} 
& = \sum_{q=1}^{N} \EXP{ \left( \partial^i f(\hat{x}_{t_{q-1}}) X^{(i)}_{t_{q-1} t_q } \right)^2 }\\
& \leqq K^2 \sum_{q=1}^{N} |t_{q} - t_{q-1}|^{2|i|\zeta+1}\\
& = K^2 \left(  \sup_{ |t-s| \leqq |\mathcal{P}|} |t-s| \right)^{2|i|\zeta} T\\
&  \to 0 \quad  (as \  |\mathcal{P}| \searrow 0),
\end{align*}
and this indicates the $L^2$ convergence.  
\end{proof}

\subsection{Proof of Theorem~\ref{ldp}}\label{pr3}
Denote by $C_{[0,T]}$ the set of the $\Real$-valued continuous functions on $[0,T]$ equipped with the uniform topology.
Let $C_{\Delta_T}$ be the set of continuous functions on $\Delta_T$, taking values in $\Real^D$, equipped with the uniform topology  
for the metric
\[
d(X,Y) := \sup_{(s,t) \in \Delta_T} \left | X_{st} - Y_{st} \right |, \quad X,Y \in C_{\Delta_T}.
\]
We use the same notation  $C_{\Delta_T}$ for different dimensions $D$, 
more specifically any one of $D=1$, $D = \max \{|i| | \ i \in I \}$, or $D = \max\{|j+k| |\  (j,k)\in J\}$.
Let $\mathcal{S}_0$ be the set of the $\mathbb{R}$-valued $\{ \mathcal{F}_t \}$-adapted simple processes on $ [0,T] \times \Omega $ and
\[
\mathcal{S} := \left \{ Z \in \mathcal{S}_0 \middle | \sup_{t \in [0,T]} |Z_t| \leqq 1  \right\}.
\]
\begin{definition}[\cite{garcia}]
Let $\{V^n \}$ be a sequence of $\mathbb{R}$-valued semimartingales on $ [0,T] \times \Omega$ . We say that the sequence is uniformly exponentially tight (UET) if for every $T >0$ and every $a >0$ there is $K_{T,a}$ such that 
\begin{equation}\label{muet1}
\limsup_{n \to \infty } \frac{1}{n} \log  \sup_{Z\in \mathcal{S} } 
\PROB{  \sup_{t \in [0,T]}\left| (Z_{-} \cdot V^n)_{t} \right| \geqq K_{T,a}  } \leqq -a,
\end{equation}
where $Z_- \cdot V$ is the It\^{o} integral of $Z$ with respect to $V$:
\[
(Z_- \cdot V)_t := \int_{0}^{t} Z_{r-}  \D V_r,
\] 
\end{definition}
For a one-dimensional Brownian motion $W$, $V^n = n^{-1/2}W$ is an example of a UET sequence; see Lemma~2.4 of \cite{garcia}.

\begin{thm}\label{mainthm}
Let $\{U^n \}$ be a UET sequence of $\mathbb{R}$-valued semimartingales and $\{ V^n\}$ a sequence of $\mathbb{R}$-valued continuous adapted processes. Assume that the sequence  $\{ (U^n, V^n) \}$ satisfies the LDP on $C_{[0,T]} \times C_{[0,T]}$ with speed $n^{-1}$ and good rate function $\tilde{J}^{\ast}$. 
Then the sequence  $\{ ( U^n, V^n, (U^n \cdot_i V^n)_{i \in I }) \}$ satisfies the LDP on
$C_{[0,T]} \times C_{[0,T]} \times C_{\Delta_T}$ with speed $n^{-1}$ and good rate function 
\begin{equation}\label{Ishsh}
\begin{split}
\tilde{J}^{\ast\ast}(u,v,x) & := \begin{cases}
\tilde{J}^{\ast} (u,v), & v \in \mathrm{BV}, \ \text{$ \forall i \in I$}, x^{(i)} = u \cdot_{i} v, \\
\infty, & \text{otherwise},
\end{cases}\\
& = \inf \left \{ \tilde{J}^{\ast} (u,v) \middle| \ u,v \in C_{[0,T]}, v \in \mathrm{BV}, \   \forall i \in I, x^{(i)} =  u \cdot_i v \right \},
\end{split}
\end{equation}
where $\mathrm{BV}$ is the set of the functions of bounded variation on $[0,T]$, $x = (x^{(i)})_{i \in I} \in C_{\Delta_T}$ and 
\[
(u \cdot_i v )_{st} := \int_{s}^{t} (u_r -u_s )^i \D v_r.
\]
\end{thm}

\begin{proof}
By the assumption and the contraction principle, $\{ ( U^n,V^n, ((U^n)^i)_{i \in I}) \} $ satisfies the LDP with good rate function 
\[
\Lambda_1 (u,v, \varphi) = \inf  \left \{\tilde{J}^{\ast} (u,v) \middle | \ \forall i \in I,  \varphi^{(i)} = u^i \right \}.
\]
 Therefore, by \cite{garcia}[Theorem~1.2], we have that $\{ (U^n, V^n, ((U^n)^i, U^n \odot_i V^n)_{i \in I} ) \}$ satisfies the LDP with good rate function
 \begin{align*}
 \Lambda_2 (u,v,\varphi, x) 
 =  \inf  \left \{\tilde{J}^{\ast} (u,v) \middle |  u,v \in C_{[0,T]},  \ v \in \mathrm{BV}, \  (\varphi^{(i)}, x^{(i)} ) =( u^i,  u \odot_i v) \right \},
 \end{align*} 
 where $(u \odot_i v)_t : = (u \cdot_i v )_{0t}$.
 Note that by the modified Chen's relation \eqref{chen1}, we have
 \[
 (u \cdot_i v )_{st} 
 = (u \odot_i v )_{t} - (u \odot_i v )_{s} - \sum_{p < i} \frac{1}{(i-p) !} (u_s - u_0)^{i-p} (u \cdot_p v )_{st}. 
 \]
 Hence, by the contraction principle again with the aid of induction, we see that 
 $\{ (U^n, V^n, (U^n \cdot_i V^n)_{i \in I} ) \}$ satisfies the LDP with good rate function \eqref{Ishsh}.
\end{proof}

\begin{thm}
\label{maincor}
Under the same conditions as in Theorem~\ref{mainthm},
the sequence  
$$\{ ( \delta U^n, (U^n \cdot_i V^n)_{i \in I}, (U^n \ast_{jk} V^n)_{(j,k)\in J}) \}$$
satisfies the LDP
on $C_{\Delta_T}  \times C_{\Delta_T} \times C_{\Delta_T} $
with speed $n^{-1}$ with good rate function 
 \begin{equation} \label{Ifun}
 \begin{split}
 \tilde{J}^{\ast\ast\ast} (\hat{x},x, \mathbf{x})
 = \inf \left\{ \tilde{J}^{\ast} (u,v) \middle|
  \begin{aligned} 
  u,v \in C_{[0,T]} &, \  v \in  \mathrm{BV},\\
  \forall i \in I, \forall (j,k) \in J,     (\hat{x},x^{(i)}, \mathbf{x}^{(jk)}) &= (\delta u, u \cdot_i v, u \ast_{jk} v) 
  \end{aligned}
  \right\},
\end{split}
\end{equation}
where $(\delta u)_{st} :=u_t - u_s$ and 
\[
(u \ast_{jk} v)_{st} := \int_s^t  (u\cdot_j v)_{sr} (u_r-u_s)^k \D v_r.
\]

\end{thm}

\begin{proof}
By Theorem~$\ref{mainthm}$ and the contraction principle, the sequence
$$\left\{ \left( U^n,V^n, (U^n \cdot_i V^n)_{i \in I}, ((U^n \odot_j V^n)  (U^n)^k)_{(j,k) \in J}\right) \right\}$$ 
satisfies the LDP  with good rate function 
\begin{align*}
 \Lambda_3 (u,v,x, \varphi) 
 = \inf  \left \{\tilde{J}^{\ast} (u,v) \middle | u,v \in C_{[0,T]}, \ v \in \mathrm{BV},  (x^{(i)}, \varphi^{(jk)}) 
 = ( u \cdot_i v, (u \odot_j v)  u^k)  \right \}.
\end{align*}
 Therefore, by \cite{garcia}[Theorem~1.2], we have that 
$$\left\{ \left(U^n, (U^n \cdot_i V^n)_{i \in I},  (U^n \circledast_{jk} V^n)_{(j,k)\in J} \right) \right\}$$
satisfies the LDP with good rate function
\begin{align*}
\Lambda_4 (u,x, \varphi) 
=  \inf  \left \{\tilde{J}^{\ast} (u,v) \middle | u,v \in C_{[0,T]}, \ v \in \mathrm{BV}, \ ( x^{(i)},
\varphi^{(jk)}) = ( u \cdot_i v, u \circledast_{jk} v ) \right \},
 \end{align*}
where $(U \circledast_{jk} V)_t = (U \ast_{jk} V)_{0t}$.
Note that by the modified Chen's relation \eqref{chen2}, we have
\begin{align*}
  (U^n \ast_{jk} V^n)_{st} 
 & = (U^n \circledast_{jk} V^n)_{t} - (U^n \circledast_{jk} V^n)_{0s}\\
& \quad - \sum_{q \leqq k } \frac{1}{(k-q)!} (U^n_{0s})^{k-q} (U^n \cdot_{j} V^n)_{0s} \otimes (U^n \cdot_{q} V^n)_{st} \\
 & \quad -  \sum_{p+q <j+k} \frac{1}{ (j-p)! (k-q)!} (U^n_{0s})^{j+k - p -q} (U^n \ast_{pq} V^n)_{st}.
\end{align*}
Hence, by the contraction principle again with the aid of induction, we see that 
$\{ ( \delta U^n, (U^n \cdot_i V^n)_{i \in I}, (U^n \ast_{jk} V^n)_{(j,k)\in J}) \}$ satisfies the LDP on $C_{\Delta_T}  \times C_{\Delta_T} \times C_{\Delta_T} $ with good rate function (\ref{Ifun}).
\end{proof}


\begin{lemma}\label{inverse}
\begin{enumerate}
\item[(i)] The ($\alpha,\beta$) rough path $\mathbb{X}$ of Theorem~\ref{ldp} has exponential integrability, i.e.,  there exists $\eta >0$ such that
\begin{equation*}
\EXP{\exp{ \left \{  \eta ||| \mathbb{X}|||^2_{(\alpha, \beta)}  \right \}} } < \infty.
\end{equation*}
\item[(ii)] Assume that the family of random variables 
$$\mathbb{X}^{\epsilon} = ( \epsilon^{1/2} \hat{X}, \epsilon^{(|i|+1)/2} X^{(i)}, \epsilon^{(|j+k|+2)/2}\mathbf{X}^{(jk)})$$
taking values in $\Omega_{(\alpha,\beta) \text{-Hld}}$ satisfies the LDP 
on $C_{\Delta_T} \times C_{\Delta_T} \times C_{\Delta_T}$
(with the uniform topology). Then, $\mathbb{X}^{\epsilon}$ satisfies the LDP on
$\Omega_{(\alpha,\beta) \text{-Hld}}$
(in the $d_{(\alpha,\beta)}$ topology) with the same good rate function.
\end{enumerate}
\end{lemma}

\begin{proof}
(i) Let $Z :=  ||| \mathbb{X}|||_{(\alpha,\beta)} $.
By the inequality (\ref{Lp1}), (\ref{Lp2}), we have that for all $p \in [2,\infty)$,
 \[
 ||X^{(i)}_{st}||_{p} 
 \leqq  C p^{(i_1+1)/2} |t-s|^{|i| \zeta + 1/2}
 \quad || \mathbf{X}^{(jk)}_{st}||_{p} \leqq C p^{(j_1+k_1+2)/2} |t-s|^{|j+k|\zeta+ 1},
 \]
and this inequality and Kolmogorov's continuity theorem (see Theorem 3.1 in \cite{FH}) imply that for $p \geqq \xi$,
\[
\left| \left| ||\hat{X}||_{\beta \text{-Hld}} \right| \right|_p 
\leqq \tilde{c} \sqrt{p}, \quad
\left| \left| ||X^{(i)}||_{|i| \beta + \alpha \text{-Hld}} \right| \right|_p 
\leqq \tilde{c} p^{(i_1+1)/2},
\]
and
\[
 \left| \left| ||\mathbf{X}^{(jk)}||_{ |j+k| \beta + 2\alpha \text{-Hld}} \right| \right|_p 
\leqq  \tilde{c} p^{(j_1+k_1+2)/2},
\]
where $\xi := \lceil \hat{\xi}^{-1} \rceil + \max_{i\in I} \lceil \xi_i^{-1} \rceil + \max_{(j,k) \in J} \lceil \xi_{jk}^{-1} \rceil$, 
$\hat{\xi} := \zeta - \beta$,  
$\xi_i: = |i|(\zeta-\beta) +(1/2-\alpha)$, $\xi_{jk} := |j+k|(\zeta-\beta) +(1-2\alpha)$,
$\tilde{c} : = \hat{c} + \max_{i \in I} c_i + \max_{(j,k) \in J} c_{jk} $, and
$$ \hat{c} := \frac{2C}{1-(1/2)^{(\hat{\xi} - \xi^{-1} )}},\quad 
c_i := \frac{2C}{1-(1/2)^{( \xi_i - \xi^{-1} )}}, \quad 
c_{jk} := \frac{2C}{1-(1/2)^{( \xi_{jk} - \xi^{-1} )}}.$$
Then 
Jensen's inequality implies that
\begin{align*}
 \left| \left| \left ( || X^{(i)}||_{|i|\beta + \alpha \text{-Hld}} \right )^{1/(|i|+1)} \right| \right|_p 
 & \leqq \left| \left| ||X^{(i)}||_{|i|\beta + \alpha \text{-Hld}} \right| \right|_p^{1/(|i|+1)}
 \leqq \tilde{c}^{1/(|i|+1)} \sqrt{p},
\end{align*}
and similarly 
\[
 \left| \left| \left ( || \mathbf{X}^{(jk)}||_{|j+k|\beta + 2\alpha \text{-Hld}} \right )^{1/(|j+k|+2)} \right| \right|_p 
 \leqq \tilde{c}^{1/(|j+k|+2)} \sqrt{p}.
\]
Therefore, we have that
\begin{align*}
    \left| \left| Z \right| \right|_p 
    \leqq c \sqrt{p}, \quad p \geqq \xi,
\end{align*}
where 
$c:=  \tilde{c} + \sum_{i \in I} \tilde{c}^{1/(|i|+1)} + \sum_{(j,k) \in J} \tilde{c}^{1/(|j+k|+2)}$. Then we have that
\begin{align*}
    \EXP{\exp\{\eta Z^2\}} 
     = \sum_{n=0}^{\infty} \frac{\eta^n}{n!}  ||Z||^{2n}_{2n} 
     \leqq \sum_{2n\leqq  \xi } \frac{\eta^n}{n!} ||Z||^{2n}_{2n} 
     + \sum_{2n >\xi } \frac{(2c^2 \eta)^n}{n!} n^n,
\end{align*}
and so taking $\eta>0$ small enough ($2c^2\eta e <1$), Stirling's formula implies the claim.
\\

(ii)
We adapt the argument of \cite{Friz}[Proposition~13.43].
By the inverse contraction principle (see Theorem~4.2.4 of \cite{DeZe}), it is sufficient to prove that $\{ \mathbb{X}^{\epsilon} \}$ is exponentially tight on $\Omega_{(\alpha,\beta) \text{-Hld}}$. By (i), 
there exists $c>0$ such that
\begin{equation*} 
\PROB{||| \mathbb{X} |||_{(\alpha^\prime,\beta^\prime)} > l} \leqq \exp{( -cl^2)}
\end{equation*}
for any $\alpha^\prime \in (\alpha,1/2)$ and $\beta^\prime \in (\beta, 1/2)$,
and this implies that for all $M>0$, there exists a precompact set
\[
K_M = \left \{ \mathbb{X} \in \Omega_{(\alpha,\beta) \text{-Hld}}\middle|\  ||| \mathbb{X}|||_{(\alpha^\prime,\beta^\prime)} \leqq \sqrt{M/c} \right \}
\]
on $\Omega_{(\alpha,\beta) \text{-Hld}}$ such that 
\begin{align*}
\epsilon \log \PROB{   \mathbb{X}^{\epsilon} \in K_M^{c} } 
&= \epsilon \log \PROB{  ||| \mathbb{X}^{\epsilon}|||_{(\alpha^\prime,\beta^\prime)} > \sqrt{\frac{M}{c}} }\\
&= \epsilon \log \PROB{  |||\mathbb{X}|||_{(\alpha^\prime,\beta^\prime)} > \sqrt{\frac{M}{c \epsilon}}  }
\leqq - M,
\end{align*}
from which we conclude the claim.
\end{proof}

The inverse contraction principle (see Theorem~4.2.4 of \cite{DeZe}) implies that $\{ \epsilon^{1/2} (W,W^{\perp})\}$ satisfies the LDP on $C^{\gamma \text{-Hld}}$ with speed $\epsilon^{-1}$ with good rate function $\tilde{I}^{\#}$ (note that $\gamma \in (0,1/2)$). By Theorem~1 in \cite{Fu23}, the map $f \mapsto \mathcal{K} f$ is continuous from $C^{\gamma \text{-Hld}}$ to $C^{\zeta \text{-Hld}}$. Then the contraction principle implies that $\{ \epsilon^{1/2} (\hat{X}^{(1)},X) = \epsilon^{1/2} (\mathcal{K}W, \rho W + \sqrt{1-\rho^2}W^{\perp} ) \}$ satisfies the LDP on $C_{[0,T]} \times C_{[0,T]}$  with speed $\epsilon^{-1}$ with good rate function 
\[
\tilde{I}^{(1)} (w,v)  = \inf \left \{ \tilde{I}^{\#}(\tilde{v}) \middle| \tilde{v} \in \mathcal{H}, (w,v) = \left( \int_0^{\cdot} \kappa(\cdot - r ) \D \tilde{v}^{(1)}_r, \rho \tilde{v}^{(1)} + \sqrt{1-\rho^2} \tilde{v}^{(2)} \right)  \right\}.
\]
Let $F_{\epsilon} :C_{[0,T]} \times C_{[0,T]} \rightarrow C_{[0,T]} \times C_{[0,T]}  \times C_{[0,T]} $ and $F:C_{[0,T]} \times C_{[0,T]} \rightarrow C_{[0,T]} \times C_{[0,T]}  \times C_{[0,T]} $ as  $F_{\epsilon} (w,v)_t := ((w_t,\epsilon^{1/2} t), v_t)$ and $F (w,v)_t := ((w_t,0), v_t)$. 
Then $F$ is continuous and $F_{\epsilon} (w^{\epsilon}, v^{\epsilon}) \to F(w,v)$ for all converging sequences $( w^{\epsilon} , v^{\epsilon}) \to (w,v)$  with $\tilde{I}^{(1)} (w,v) < \infty$. 
Hence the extended contraction principle \cite{Pu}[\text{Theorem~2.1}] implies that $\{ \epsilon^{1/2} (\hat{X}, X) \}$ satisfies the LDP on $C_{[0,T]} \times C_{[0,T]}  \times C_{[0,T]}$ with speed $\epsilon^{-1}$ with good rate function
\[
\tilde{J}^{\ast} (u, v) := \inf\left \{ \tilde{I}^{\#} (\tilde{v} ) \middle| \tilde{v} \in \mathcal{H}, \ (u,v) = \mathbb{K} (\tilde{v}) \right\}.
\]
As mentioned earlier, $\{ X^\epsilon = \epsilon^{1/2} X \}$ is UET by Lemma~2.4 of \cite{garcia} with
$n = \epsilon^{-1}$. Therefore, by Lemma~\ref{inverse} (ii) and Theorem~\ref{maincor} (regarding $n = \epsilon^{-1}$, $U=\hat{X}$ and $V=X$), we have proved Theorem~$\ref{ldp}$.

\appendix
\section{A lemma from rough path theory}

\begin{lemma}[\cite{inahama} Proposition~1.6]\label{itten}
Let $\omega$ be a control function, i.e., 
\[
\omega(s,u) + \omega(u,t) \leqq \omega(s,t), \quad 0 \leqq s \leqq u \leqq t \leqq T,
\] 
and $\mathcal{P} = \{s = t_0 < t_1 <... < t_N = t\}$ be a partition on $[s,t]$ ($N\geqq 2$). Then there exists an integer $i$ ($1 \leqq i \leqq N$) such that 
\begin{equation*}
\omega(t_{i-1},t_{i+1}) \leqq \frac{2}{N-1} \omega(s,t).
\end{equation*}
\end{lemma}
\begin{proof}
By the definition of $\omega$, we have
\[
\sum_{p=1}^{N-1} \omega(t_{i-1}, t_{i+1}) = \sum_{i:\text{odd}} \omega(t_{i-1}, t_{i+1}) + \sum_{i:\text{even}} \omega(t_{i-1}, t_{i+1}) \leqq 2 \omega (s,t).
\]
Therefore, there exists such $i$ that satisfies the desired inequality.
\end{proof}

\section{Proof of Theorem \ref{FZ}}\label{C}

\begin{proof}
For brevity, let $ \sigma := \sigma(S_0)$. 
By Theorem~$\ref{thm:36}$ and the contraction principle, $t^{H-1/2} \bar{S}_{t}$ satisfies the LDP with speed $t^{-2H}$ with good rate function 
\begin{align*}
 \tilde{J} (\tilde{s})
 := \inf \left \{ \tilde{I}^{\#}(\tilde{v}) \middle| \tilde{v} \in \mathcal{H}, \ \tilde{s} =   \left (\sigma \int  f(\widehat{\mathbb{L}\circ \mathbb{K}} (\tilde{v}) \D \mathbb{L} \circ \mathbb{K} (\tilde{v}) \right)^{(1)}_{01} \right \}.
\end{align*}
Let $\tilde{v} = (h^1,h^2) \in \mathcal{H} (\mathbb{R}) \times \mathcal{H}(\mathbb{R})$.
Then 
\begin{align*}
\tilde{s} 
& =  \left( \sigma \int f(\widehat{\mathbb{L} \circ \mathbb{K}} (\tilde{v})) \mathbb{L} \circ \mathbb{K} (\tilde{v}) \right)^{(1)}_{01} \\
& = \sigma  \int_0^1 f \left (\int_0^t \kappa_H (t-r) \dot{h}^1_r \D r, 0 \right) \D \left ( \rho h^1_t + \sqrt{1-\rho^2} h^2_t \right)\\
& = \rho \sigma \int_0^1 f \left (\int_0^t \kappa_H (t-r) \dot{h}^1_r \D r ,0\right) \D h^1_t 
+  \sqrt{1-\rho^2} \sigma \int_0^1 f \left (\int_0^t \kappa_H (t-r) \dot{h}^1_r \D r ,0\right) \D h^2_t,
\end{align*}
and so
\begin{equation}\label{cond2}
\frac{\tilde{s} -  \rho \sigma  \int_0^1 f \left (\int_0^t \kappa_H  (t-r) \dot{h}^1_r \D r,0 \right) \D h^1_t }{\sqrt{1-\rho^2}}  = \sigma \int_0^1 f \left (\int_0^t \kappa_H (t-r) \dot{h}^1_r \D r,0 \right) \D h^2_t.
\end{equation}
Fix $h_1$, and minimize $\frac{1}{2} || \tilde{v} ||^2_{\mathcal{H}(\mathbb{R}^2) }$ with respect to $h_2 \in \mathcal{H}(\mathbb{R})$ under the condition ($\ref{cond2}$). 
Let $\tilde{h}$ be the minimizer. Take $\epsilon > 0$ and $\hat{h} \in \mathcal{H}(\mathbb{R})$, and consider $\tilde{h} + \epsilon \hat{h}$. Because $\tilde{h}$ satisfies the condition ($\ref{cond2}$), 
\begin{equation}\label{orth}
\int_0^1 f \left (\int_0^t \kappa_H (t-r) \dot{h}^1_r \D r,0 \right) \D \hat{h}_t =0.
\end{equation}
Because $\tilde{h}$ is the minimizer, we have
\[
\left.\frac{\D}{\D \epsilon}\right|_{\epsilon=0} \frac{1}{2} \int_0^1 ( \dot{\tilde{h}}_r + \epsilon \dot{\hat{h}}_r)^2  \D r =0,  \ \ 
\text{i.e.,} \ \ 
\int_0^1 \dot{\tilde{h}}_r \dot{\hat{h}}_r \D r =0,
\]
for any $\hat{h}$ with \eqref{orth}.
Therefore, there exists $c \in \mathbb{R} $ such that 
\[
\dot{\tilde{h}} = c f\left (\int_0^{\cdot} \kappa_H  (\cdot-r) \dot{h}^1_r \D r,0 \right).
\]
Hence 
\[
\frac{\tilde{s} -  \rho \sigma \int_0^1 f \left (\int_0^t \kappa_H (t-r) \dot{h}^1_r \D r,0 \right) \D h^1_t }{\sqrt{1-\rho^2}} 
= c \sigma \int_0^1 f^2 \left (\int_0^t \kappa_H (t-r) \dot{h}^1_r \D r,0 \right) \D t,
\]
and we conclude that 
\begin{align*}
 \tilde{J}(\tilde{s}) = \tilde{I}^{\#}(\tilde{v} ) 
 = \frac{1}{2} \int_0^1 |\dot{h}^1_r|^2 \D s + \frac{ \left \{\tilde{s} - \rho \sigma \int_0^1 f \left (\int_0^t \kappa_H (t-r) \dot{h}^1_r \D r,0 \right) \D h^1_t \right\}^2 }{2 (1 - \rho^2) \sigma^2 \int_0^1 f^2 \left(\int_0^t \kappa_H (t-r) \dot{h}^1_r \D r,0 \right) \D t },
\end{align*}
which is the claim.
\end{proof}

\renewcommand\refname{\mbox{\normalfont\Large\bfseries Bibliography}}

\end{document}